\documentclass[11pt]{article}

\usepackage[english]{babel}
\usepackage[a4paper,margin=2.5cm]{geometry}
\usepackage[colorlinks=true, allcolors=blue]{hyperref}
\usepackage{amsmath}
\usepackage{amsfonts}
\usepackage{amssymb}
\usepackage{amsthm}
\usepackage{graphicx}
\usepackage{mathtools}
\usepackage{mathrsfs}
\usepackage[utf8]{inputenc}
\usepackage{xcolor}
\usepackage{verbatim}
\usepackage{float}
\usepackage{authblk}

\numberwithin{equation}{section}
\numberwithin{figure}{section}

\newtheorem{defi}{\textbf{Definition}}[section]
\newtheorem{lemma}[defi]{\textbf{Lemma}}
\newtheorem{coro}[defi]{\textbf{Corollary}}
\newtheorem{thm}[defi]{\textbf{Theorem}}
\newtheorem{prop}[defi]{\textbf{Proposition}}
\newtheorem{rmk}[defi]{\textbf{Remark}}

\newcommand{\Ac}{\mathcal{A}}
\newcommand{\Bc}{\mathcal{B}}
\newcommand{\Dc}{\mathcal{D}}
\newcommand{\Ec}{\mathcal{E}}
\newcommand{\Fc}{\mathcal{F}}

\newcommand{\Pc}{\mathcal{P}}
\newcommand{\Sc}{\mathcal{S}}
\newcommand{\Vc}{\mathcal{V}}
\newcommand{\Wc}{\mathcal{W}}
\newcommand{\Xc}{\mathcal{X}}
\newcommand{\Yc}{\mathcal{Y}}
\newcommand{\Zc}{\mathcal{Z}}

\newcommand{\Eb}{\mathbb{E}}

\newcommand{\Pb}{\mathbb{P}}
\newcommand{\Rb}{\mathbb{R}}
\newcommand{\Zb}{\mathbb{Z}}
\newcommand{\Cb}{\mathbb{C}}

\newcommand{\dist}{\mathrm{dist}}

\newcommand{\sausage}{D}
\newcommand{\annulus}{\Ac}

\newcommand{\fr}{\mathrm{fr}}
\newcommand{\eps}{\varepsilon}
\newcommand{\wh}{\widehat}
\newcommand{\wt}{\widetilde}
\newcommand{\ol}{\overline}
\newcommand{\ind}[1]{{\mathbf{1}{\{#1\}}}}

\newcommand{\intermediatepart}{\zeta}
\newcommand{\initialconfig}{\beta}

\begin{document}
\title{Convergence in natural parametrization of random walk frontier}
\author[1]{Yifan Gao\thanks{yifangao@cityu.edu.hk}}
\author[2]{Xinyi Li\thanks{xinyili@bicmr.pku.edu.cn}}
\author[2]{Runsheng Liu\thanks{liurunsheng@pku.edu.cn}}
\author[2]{Xiangyi Liu\thanks{liuxiangyi@stu.pku.edu.cn}}
\author[3]{Daisuke Shiraishi\thanks{shiraishi@acs.i.kyoto-u.ac.jp}}
\affil[1]{City University of Hong Kong}
\affil[2]{Peking University}
\affil[3]{Kyoto University}
\maketitle
\vspace{-0.5cm}
\begin{abstract}
In this paper, we show that the frontier of planar random walk converges weakly under natural parametrization to that of planar Brownian motion. As an intermediate result, we also show the convergence of the renormalized occupation measure.
\end{abstract}
	
\section{Introduction}\label{sec1:Introduction}
The \emph{Brownian frontier} is arguably one of the most basic examples to understand the multi-fractal structure of two-dimensional Brownian motion, and has been studied in depth for a long time. The Hausdorff dimension of Brownian frontier was conjectured to be $4/3$ by Mandelbrot \cite{M82} in 1982. Based on a moment argument, Lawler \cite{L96} showed that the Hausdorff dimension of Brownian frontier is $2-\xi(2)$, where $\xi(\cdot)$ is the so-called \textit{disconnection exponent} (we refer readers to Section \ref{sec2.2:prelim} for a brief introduction). The values of such exponents are first predicted by Duplantier and Kwon \cite{DK88}, and later rigorously determined by Lawler, Schramm and Werner in a series of pioneer works \cite{LSW01,LSW01a,LSW01b,LSW02,LSW02a} using Schramm-Loewner evolution (SLE) techniques. Moreover, they further show in \cite{LSW03}  that the outer boundary of Brownian excursion is distributed as some variant of $\mathrm{SLE}_{8/3}$. 
\par
Based on the convergence from random walk to Brownian motion, it is quite natural to ask, which kind of convergence can be established for their outer boundaries. As an intermediate result of  \cite{VCL16} (which actually focuses on loop soups), the \emph{random walk frontier} converges under the Hausdorff metric to the Brownian frontier under a proper setup. In this work, we further prove the convergence in the stronger sense of \emph{natural parametrization}; see Theorem \ref{thm:convergence under np}. For this, we first establish the convergence of the renormalized occupation measure of random walk frontier following the work \cite{HLLS22,GLPS23} on cut points. It is based on a standard $L^{2}$-approximation, which is a powerful method and has been widely used in different context, e.g.\ percolation \cite{GPS13,HLS22}, loop-erased random walk \cite{LV16,LS24}, and SLE \cite{L15,LR15,Z22}. Viewing random walk frontier as a continuous random curve, we then show the convergence modulo time-parametrization and finally under natural parametrization, by showing tightness and the uniqueness of subsequential limit in both cases.
\par

In the following, we will state our results on convergence under natural parametrization and convergence of occupation measure in Sections~\ref{sec1.1:Convergence under natural parametrization 1} and~\ref{sec1.2:occupation}, respectively.

\subsection{Convergence under natural parametrization} \label{sec1.1:Convergence under natural parametrization 1}

Suppose $A \subseteq \mathbb{C}$ is a bounded set. The \emph{frontier} of $A$, denoted by $\fr(A)$, is defined as the boundary of the unbounded connected component of $\mathbb{C}\setminus A$. 
For any $\delta\in [0,2]$, its \emph{$\delta$-Minkowski content} is defined as the following limit (provided the existence) 
\begin{equation}\label{eq:def of mink}
    \mathrm{Cont}_{\delta}(A):=\lim_{\eps\to 0} \eps^{\delta-2}\mathrm{Area}(A^\eps),\ \mathrm{where}\ A^\eps:=\{z\in\Cb: \dist(z,A)\le\eps\}.
\end{equation}
    
Let $W$ be a Brownian motion starting from the origin and denote by $T_0 = \inf\{t:|W(t)|=1\}$ the time when $W$ first hits the unit circle.
Let $\wt\gamma$ be the curve that traces $\fr(W[0,T_0])$ from $W(T_0)$ in the counterclockwise direction and ends at $W(T_0)$ (it is indeed a non-simple loop surrounding the origin \cite{Q21}). In \cite{LR15}, it is shown that the $(1+\kappa/8)$-Minkowski content of SLE$_\kappa$ exists. Since the Brownian frontier is a variant of SLE$_{8/3}$ (see \cite{LSW03,Q21} for details), locally, the Brownian frontier is absolutely continuous with respect to SLE$_{8/3}$. Therefore, the $4/3$-Minkowski content for the Brownian frontier exists. And the natural parametrization of $\widetilde\gamma$ is the one such that 
\begin{equation}
    \mathrm{Cont}_{4/3}(\widetilde\gamma[0,t])=t,\ 0\le t\le \mathrm{Cont}_{4/3}(\widetilde\gamma)\qquad\mbox{ a.s.}
\end{equation}
Let $\lambda_n$ be a simple random walk on $\Zc_n=e^{-n}\Zb^2$ and denote its first exiting time of the unit disk by $\tau_0$. We consider the continuous curve $\gamma_n$ that traces the random walk frontier $\fr(\lambda_n[0,\tau_0])$ from $\lambda_n(\tau_0)$ in the counterclockwise direction and ends at $\lambda_n(\tau_0)$, traversing each edge with time $c_1 e^{-4n/3}$ (here, the constant $c_1$, which appears in \eqref{eq:discrete occupation measure}, will be defined in Theorem \ref{thm:one-point1}). 

Let $\Pc = \{\gamma:[0,t_\gamma]\to\Rb^2\}$ denote the collection of all the curves in $\Rb^2$, where $t_\gamma$ is the time duration of $\gamma$. Equip $\Pc$ with the natural parametrization metric formally defined by 
\begin{equation}\label{eq:def rho}
    \rho(\gamma,\gamma'):=\inf_{\alpha}\Big[ \sup_{0\leq t \leq t_\gamma} |\alpha(t)-t|+ \sup_{0\leq t \leq t_\gamma}\big|\gamma'(\alpha(t))-\gamma(t)\big|\Big],
\end{equation}
where the infimum is taken over all continuous bijections $\alpha:[0,t_\gamma]\to[0,t_{\gamma'}]$. 
We now briefly introduce our main result. 

\begin{thm}\label{thm:convergence under np}
	$\gamma_n$ converges to $\widetilde{\gamma}$ under the natural parametrization metric $\rho$.
\end{thm}
For the proof of Theorem \ref{thm:convergence under np}, we will first show the convergence modulo time-parametrization. To this end, we make use of \cite{AB99} to show tightness, and the uniqueness of subsequential limit will follow from the convergence under Hausdorff metric (see \cite{VCL16}). This convergence, combined with the convergence of occupation measure (see Theorem \ref{thm:main result 0}), will conclude the proof.

\subsection{Convergence of occupation measure}\label{sec1.2:occupation}
Recall from \eqref{eq:def of mink} the definition of Minkowski content. One can define an almost surely regular non-atomic finite Borel measure $\nu$ by letting 
\begin{equation}
    \nu (\cdot):=\mathrm{Cont}_{4/3}(\cdot\cap\widetilde\gamma).
\end{equation}
As for random walks, we can define the occupation measure as follows:
\begin{equation}\label{eq:discrete occupation measure}
	\nu_n:=c_1e^{-{4n/3}}\sum_{x\in \mathcal{Z}_n \cap \fr(\lambda_n) }\delta_x,
\end{equation}
where $\delta_x$ is a unit point mass at $x$ and $c_1$ is the positive constant defined in Theorem~\ref{thm:one-point1}. The following theorem establishes the convergence of occupation measure. 
\begin{thm} \label{thm:main result 0}
The law of $\nu_n$ converges to that of $\nu$ with respect to the topology of weak convergence of finite measures.
\end{thm}
The proof of Theorem \ref{thm:main result 0} follows a strategy similar to \cite{GLPS23}. We are able to obtain one-point and two-point estimates for frontier Green's function (see Section \ref{sec3:Continuum frontier disks and frontier Green's function} for a precise definition). And we can also deduce sharp estimates for frontier-disk events, which allows us to control the $L^2$-distance of the occupation measure induced by frontier and frontier-disk. Finally, we control the cross-term (i.e.\ the $L^2$-distance between the occupation measure induced by discrete and continuum frontier-disk) by using Skorokhod embedding to couple random walk and Brownian motion.

\subsection{Organization of the paper}\label{sec1.3:Organization of the paper}

In Section \ref{sec2:notation}, we introduce some basic notation related to Brownian paths, Brownian measures and discrete paths, and review some preliminary results for disconnection exponents, separation lemmas and coupling techniques. 

Sections \ref{sec3:Continuum frontier disks and frontier Green's function}-\ref{sec5:Convergence of occupation measure} are devoted to proving Theorem~\ref{thm:main result 0}. Section \ref{sec3:Continuum frontier disks and frontier Green's function} introduces the frontier-disk events and the frontier Green's function, while Section \ref{sec4:Moment estimates of frontier points and frontier disks} focuses on the relations between frontier points and disks. In Section \ref{sec5:Convergence of occupation measure}, we complete the proof of Theorem~\ref{thm:main result 0}, the convergence of occupation measure, which also plays an important role in deriving the convergence under natural parametrization. 

In Section \ref{sec6:Convergence under natural parametrization}, we derive Theorem \ref{thm:convergence under np} by proving a weaker convergence (convergence under reparametrization) first and making use of the convergence of occupation measure.

\bigskip

\noindent {\bf Acknowledgements}: Xinyi Li wishes to thank Greg Lawler for suggesting this problem and for helpful discussions. Yifan Gao is supported by National Key R\&D Program of China (No.\ 2023YFA1010700). Xinyi Li, Runsheng Liu and Xiangyi Liu are supported by National Key R\&D Program of China (No.\ 2021YFA1002700 and No.\ 2020YFA0712900). Daisuke Shiraishi is supported by JSPS Grant-in-Aid for Scientific Research (C) 22K03336, JSPS Grant-in-Aid for Scientific Research (B) 22H01128
and 21H00989.

\section{Notation and Preliminaries}\label{sec2:notation}
In this section, we lay the groundwork for the subsequent proof. In Section \ref{sec2.1:notation2}, we introduce the necessary notation and basic concepts that will be used extensively throughout this paper. In Section \ref{sec2.2:prelim}, we present several fundamental results concerning Brownian motion and SRW, which will be useful in analyzing the non-disconnection event and coupling the discrete and continuous paths.
\subsection{Notation}\label{sec2.1:notation2}
\paragraph{Conventions.}
Throughout this paper, we use the notation $c,C,u$ to represent positive constants whose values may vary across different lines. In contrast, constants such as $c_1,c_2$ are reserved for fixed positive constants. 
Suppose $a_n$ and $b_n$ are two positive sequences. We will use the following notation:
\begin{itemize}
	\item $a_n\lesssim b_n$ or $a_n=O(b_n)$ means there exists a positive constant $c$ such that $a_n\le c b_n$ for all $n$.
	\item $a_n\asymp b_n$ indicates that $a_n$ and $b_n$ are comparable. In other words, there exist positive constants $c_1$ and $c_2$ such that $c_1 b_n\le a_n\le c_2 b_n$ for all $n$.
	\item $a_n\simeq b_n$ holds if there exists a positive constant $u$ such that $a_n=b_n(1+O(e^{-un}))$.
\end{itemize}
We add subscripts to these symbols to indicate that the associated constants may depend on the introduced subscripts.
\paragraph{Sets in $\mathbb R^2$.}
For $z\in \Rb^2$ and $r>0$, we write 
\begin{align*}
	D_r(z)=D(z,r):=\{ x \in \Rb^2 : \left | x-z \right | < r \}\ \ \mathrm{and}\ \ 
    B_r(z)=B(z,r):=D_r(z)\cap \Zb^2.
\end{align*}
To simplify notation in exponential scaling scenarios, we introduce
\begin{equation*}
	\mathcal{D}_r(z):=D_{e^r}(z)\ \ \mathrm{and}\ \ \mathcal{B}_r(z):= D_{e^r}(z)\cap \Zb^2. 
\end{equation*}
When $z=0$, we omit $z$ and simply write $\Dc_r$ and $\Bc_r$. The symbol $\mathcal{D}$ specifically denotes the unit disc $D(0,1)$. For $A\subseteq \Rb^2$, the $\delta$-sausage of $A$ is defined as 
\[
\sausage(A,\delta):=\{x\in\Rb^2:\dist(x,A)\le \delta\}.
\]
The annulus centered at $x$ with radii $r$ and $R$ is denoted by
\[
\annulus(x,r,R) := \{y\in \Rb^2 : r\leq |y-x|\leq R\}.
\]
We also give notation in the discrete setting. If $A\subseteq \Zb^2$, define the outer boundary of $A$ as
\begin{align*}
	\partial A := \{x\in \Zb^2\setminus A: \exists\, y\in A\ \mathrm{such\ that}\ |x-y|=1\}.
\end{align*}

\paragraph{Paths.}
Let $\Pc$ denote the collection of all continuous curves in $\Rb^2$. For a curve $\gamma \in \Pc$ and a subset $D \subseteq \Rb^2$, the hitting time is defined as 
\[    
H_D(\gamma):=\mathrm{inf}\{t\ge 0:\gamma(t) \in D\}.
\]
For a Brownian motion W and a simple random walk S, we respectively denote 
\[T_r:= H_{\partial \mathcal{D}_r}(W)\ \ \mathrm{and}\ \ \tau_n:=H_{\partial \mathcal{B}_n}(S).\]
The concatenation of paths $\gamma[0,t]$ and $\gamma'[0,t']$ (with $\gamma(t)=\gamma'(0)$) is defined as 
\[
\gamma\oplus\gamma'(s) := \gamma(s)\ind{0\le s\le t}+\gamma'(s-t)\ind{t< s\le t+t'}.
\]
A discrete path $\lambda$ in the integer lattice $\Zb^2$ is defined as an ordered sequence of vertices
\[
\lambda =\big(\lambda(0),\lambda(1),\ldots,\lambda(\mathrm{len}(\lambda))\big),
\]
where $\mathrm{len}(\lambda)$ denotes its length, and adjacent vertices satisfy $|\lambda(j)-\lambda(j-1)|=1$. To interpret $\lambda$ as an element of the continuous path space $\Pc$, we linearly interpolate between consecutive vertices. Each edge is traversed in $1/2$ units of time, resulting in a total duration of $t_\lambda = \mathrm{len}(\lambda)/2$. For interpolated paths $\gamma$ derived from discrete paths, the hitting time $H_D(\gamma)$ is interpreted as $H_{D\cap \Zb^2}(\gamma)$. This slight abuse of notation ensures compatibility with earlier definitions when applied to lattice-based paths.

\paragraph{Path measures.}
For any finite measure $\mu$, its total mass is denoted by $\left \| \mu\right \| $. We define the normalized measure as $\widehat{\mu}=\mu / \left \| \mu\right \| $. If $g$ is a function defined on the support of $\mu$, then the integral of $g$ with respect to $\mu$ can be written as $\mu(g)$ or $\mu[g]$.

Let $D\subseteq\Cb$ be a domain with piecewise analytic boundary $\partial D$. We will use the following (Brownian) path measures in $D$ frequently. We refer to \cite[Chapter 5]{L08} for details on the construction of these measures.

\medskip

\noindent\textbf{(1) Interior-to-interior path measure.} For $z,w\in D$, let $\mu_D(z,w,t)$ be the probability measure on Brownian bridges from $z$ to $w$ of duration $t$, conditioned on staying in $D$. The interior-to-interior measure $\mu_{z,w}^D$ is then obtained by integrating over time.
\[
\mu_{z,w}^D:=\int_0^\infty \mu_D(z,w,t)dt.
\]

\medskip

\noindent\textbf{(2) Interior-to-boundary path measure.} For $z\in D$, let $W$ be the Brownian motion started from $z$ and $T_{\partial D}$ be the hitting time of ${\partial D}$ by $W$.
Let $\mu^D_{z,\partial D}$ be the probability measure induced by $(W(s))_{0\le s \le T_{\partial D}}$. 
Moreover, for $w\in\partial D$, we use  $\mu^D_{z,w}$ (with slight abuse of notation) to denote the interior-to-boundary path measure, given by $(W(s))_{0\le s \le T_{\partial D}}$ conditioned on $W(T_{\partial D})=w$.

\medskip

\noindent\textbf{(3) Boundary-to-boundary path measure.} For $z,w \in \partial D$, the boundary-to-boundary path measure is defined via 
\[
\mu^D_{z,w}:=\lim_{\varepsilon \to 0}\varepsilon^{-1}\mu^D_{z+\varepsilon n_z,w},
\]
where $n_z$ is the inward normal vector at $z$. By integration, we extend this to measures $\mu_{z,\partial D}^D$,\ $\mu_{\partial D,w}^D$ and $\mu_{\partial D,\partial D}^D$. 
These measures are also called excursion measures in $D$. As for random walks, one can also define RW path measures $\nu^D_{z,w}$, $\nu^D_{z,\partial D}$, $\nu^D_{\partial D,w}$ and $\nu^D_{\partial D,\partial D}$ by summing over the total mass of the corresponding paths. 

\subsection{Preliminaries}\label{sec2.2:prelim}
\paragraph{Intersection and disconnection exponents.} Let $W^0,\ldots,W^k$ be independent planar Brownian motions starting from the origin. For $r>0$, define the random variable:
\begin{equation}
    Z_r=Z_r(W^1,\ldots,W^k) = \Pb\big(\,W^0[T_0,T_r]\cap(W^1[T_0,T_r]\cup\cdots\cup W^k[T_0,T_r]) = \emptyset \,\big|\,W^1,\ldots,W^k\,\big).
\end{equation}
Note that $Z_{r}$ is a random variable which is a function of $W^1,\ldots,W^k$.
For $\lambda>0$, a standard subadditivity argument shows that the limit
\begin{equation}
\xi(k,\lambda):=-\lim_{r\to \infty} \frac{\log  \Eb(Z_r^\lambda)}{\log r}
\end{equation}exists and $\xi(k,\lambda)$ are called the {\it intersection exponent}. 
Taking $\lambda\to 0$, the {\it disconnection exponent} $\xi(k)$ is defined to be the limit
\begin{equation}
\xi(k):=-\lim_{r\to \infty} \frac{\log \Pb(Z_r>0)}{\log r} .
\end{equation}
In \cite{LSW01a,LSW01b,LSW02,LSW02a}, the intersection exponents are analytically derived through the expression
\begin{equation}\label{eq:intersection exponents}
    \xi(k,\lambda) = \frac{(\sqrt{24k+1}+\sqrt{24\lambda+1}-2)^2-4}{48},\ k\in\mathbb N^*,\ \lambda>0.
\end{equation}
The disconnection exponent $\xi(k)$ emerges naturally as the limit of $\xi(k,\lambda)$ when $\lambda\to 0^+$,  characterizing the decay rate of path disconnection probabilities. 
For the special case $k=1$, the exponent reduces to $\xi(1)=1/4$, which governs the asymptotic behavior of one-arm non-disconnection probabilities. This result is stated as the lemma below; see e.g.\ \cite{L98} for details.
\begin{lemma}[One-arm non-disconnection estimates] \label{one-arm disconnection exponent}Let $r<R$, and consider a Brownian motion $W$ starting from $x_0 \in\partial\Dc_r$ and stopped upon hitting $\partial \Dc_R$. Denote by $\widetilde{D}_{r,R}$ the event that $W$ does not disconnect $\partial\Dc_r$ and $\infty$. Then, 
\begin{equation}\label{eq:one-arm event}
\Pb^{x_0}\big(\widetilde{D}_{r,R}\big)\asymp e^{-(R-r)/4}.
\end{equation}\end{lemma}
Here, for Brownian motions $W_1,\ldots,W_k$ starting from $x_1,\ldots,x_k$, we denote the associated probability measure by $\Pb^{x_1,\ldots,x_k}$. Its expectation is denoted by $\Eb^{x_{1}, \ldots , x_{k}}$. We use $\Pb$ and $\Eb$ when $x_{i} = 0$ for all $i$. With slight abuse of notation, we extend these symbols $\Pb^{x_1,\ldots,x_k}$ and $\Eb^{x_{1}, \ldots , x_{k}}$ to also represent the probability measure and its expectation corresponding to SRWs starting from $x_1,\ldots,x_k$. We also write $\Pb$ and $\Eb$ when $x_{i} = 0$ for all $i$.
\paragraph{Two-arm non-disconnection probabilities.} 
\textbf{(1) Continuous framework.} 
Let $\widetilde\Gamma_r$ denote the collection of paths $\initialconfig$ satisfying $\initialconfig(0)=0$, $|\initialconfig(t_{\initialconfig})|=e^r$ and $|\initialconfig(t)|<e^r$ for all $0\le t<t_{\initialconfig}$. The set of pairs of non-disconnecting paths is defined as
\begin{equation}
	\widetilde{\mathcal{Y}}_r:=\big\{ \ol\initialconfig=(\initialconfig^1,\initialconfig^2)\in \widetilde{\Gamma}_r\times \widetilde{\Gamma}_r: 0\in \fr(\initialconfig^1\cup\initialconfig^2)\big\}.
\end{equation}
Write $\widetilde{\mathcal{Y}}=\widetilde{\mathcal{Y}}_0$. 
For a pair $\ol\initialconfig\in \widetilde{\mathcal{Y}}$ with endpoints $\ol x=(x_1,x_2)$ on $\partial\mathcal{D}$, consider two independent Brownian motions $W^1$, $W^2$ starting from $x_1$, $x_2$ respectively. For $n>0$ and $i=1,2$, define the concatenated paths under the scaling $e^{-n}$ as
$$\widehat\initialconfig^i_n:=e^{-n}\big(\initialconfig^i\oplus W^i[0,T_n]\big),\ i=1,2,$$
where $\oplus$ denotes path concatenation.  
The non-disconnection event at scale $n$ is 
\begin{equation*}
	\widetilde{D}_n(\ol\initialconfig):=\big\{ \widehat\initialconfig_n = (\widehat\initialconfig^1_n,\widehat\initialconfig^2_n)\in \widetilde{\mathcal{Y}} \big\}.
\end{equation*}
In \cite{L96}, Lawler establishes the following up-to-constants estimate for the non-disconnection probability:
\begin{equation} \label{eq:non-disconnection probability}
\Pb^{x_1,x_2}\big( \widetilde D_n(\ol\initialconfig)\big) \asymp_{\ol\initialconfig} e^{-\alpha n},
\end{equation}
where $\ol\initialconfig\in\wt\Yc$ and $\alpha=\xi(2)$ corresponds to the two-arm disconnection exponent, which is known from \eqref{eq:intersection exponents} to equal 2/3. Additionally, we have the sharp estimate for $\widetilde{D}_n(\ol\initialconfig)$:
\begin{equation}\label{eq:259}
    \Pb^{x_1,x_2}\big( \widetilde D_n(\ol\initialconfig)\big) \simeq  q e^{-\alpha n},
\end{equation}
where $q=q(\ol\initialconfig)$ is a constant depending on $\ol\initialconfig$. For details, we refer to Section 6 of \cite{LSW02b}. 

The following lemma extends these results to Brownian excursions; see e.g.\ \cite{LSW02b,LW04}.
\begin{lemma} \label{thm:total mass}
	Let $W^1$ and $W^2$ be independent Brownian motions, $s < r$, and $\sigma_s^i$ denote the last exit time of $\Dc_s$ by $W^i$ before $T_r^i$. Then,
	\begin{equation}
		\Pb(W^1[\sigma_s^1,T_t^1]\cup W^2[\sigma_s^2,T_t^2]\mathrm{\ does\ not\ disconnect\ }\partial\Dc_s\mathrm{\ from\ }\partial\Dc_t )\asymp (t-s)^2e^{-\alpha(t-s)}.
	\end{equation} 
\end{lemma}
\textbf{(2) Discrete framework.} 
Let $\Gamma_m$ denote the set of paths $\initialconfig$ where $\Pb\{ S[0, \tau_m]=\initialconfig \}>0$ for the simple random walk $S$ starting from $0$. Define the collection of pairs of discrete non-disconnecting paths as
\begin{equation}\label{eq:Yl}
	\mathcal{Y}_{m}:=\, \big\{\ol{\initialconfig}=(\initialconfig^{1},\initialconfig^{2})\in\Gamma_m\times\Gamma_m: 0\in \fr(\initialconfig^1\cup\initialconfig^2)\big\}.
\end{equation}
For $m>l>0$ and an initial configuration $\ol{\initialconfig}\in \mathcal{Y}_l$, the discrete non-disconnection event $D_{m}(\ol{\initialconfig})$ is defined by
\begin{equation*}\label{eq:D}
	D_{m}(\ol{\initialconfig}):= \big\{0\in \fr(\initialconfig^1\cup S^{1}[0,\tau_{m}]\cup \initialconfig^2 \cup S^{2}[0,\tau_{m}])\big\},
\end{equation*}
where $S^1$ and $S^2$ are independent SRWs starting from $x_1$ and $x_2$, the endpoints of $\initialconfig^1$ and $\initialconfig^2$. It has been shown in \cite{LP00} that 
\begin{equation} \label{eq:rw-disc-exp}
	\Pb^{x_1,x_2} \big(D_{m}(\ol \initialconfig)\big)  \asymp_{\ol\initialconfig} e^{-\alpha m}.
\end{equation}
Similar to \eqref{eq:259}, we also present the following sharp estimate. Let $D_m$ denote the event that $0\in\fr(S_1[0,\tau_m]\cup S_2[0,\tau_m])$, where $S_1$ and $S_2$ are independent SRWs starting from the origin. Then, we have the sharp estimate
\begin{equation}\label{sharp estimate for D_m}
    \Pb^{0,0} (D_{m})  \simeq q e^{-\alpha m},
\end{equation}
where $q$ is a constant. Although this sharp estimate is not stated explicitly in the literature, its proof is now more or less standard (in principle, given a sharp estimate \eqref{eq:259} for the non-disconnection probability associated with Brownian motion, we then use the Skorokhod embedding to translate the result to SRW \cite{S12}), and thus omitted for brevity.

\paragraph{Separation lemmas.} 
The key ingredient to establish sharp estimates \eqref{eq:259} and \eqref{sharp estimate for D_m} is the separation lemma. We follow the methodology of \cite{L96}. 
For $\ol\initialconfig\in\wt\Yc$, define the non-disconnection quality $\widetilde{\Delta}_n(\ol{\initialconfig})$ as the maximal radius $\delta >0$ for which attaching disks of radius $\delta$ to the endpoints of scaled Brownian paths preserves connectivity between $0$ and $\infty$:
\begin{equation} \label{eq:disc-qua}
	\widetilde{\Delta}_n(\ol{\initialconfig})=\sup\big\{\delta>0:0\in\fr\big(\widehat{\initialconfig}_n^1\cup\widehat{\initialconfig}_n^2\cup D(x_n^1,\delta)\cup D(x_n^2,\delta)\big)\big\},
\end{equation}
where $x_n^i$ denotes the endpoint of the scaled path $\widehat{\initialconfig}_n^i$. The following separation lemma is established in Lemma 3.2 of \cite{L96}. 
\begin{lemma}[Separation lemma, continuous case]\label{lem:separation lemma}
	There exists a universal constant $c>0$, such that for any initial configuration $\ol{\initialconfig} \in \widetilde{\mathcal{Y}}$,
	\begin{equation}
		\Pb^{x_1,x_2}\big(\widetilde{\Delta}_n(\ol{\initialconfig})\ge {1}/{4}\big)\ge c\,\Pb^{x_1,x_2}\big(\widetilde{D}_n(\ol{\initialconfig})\big).
	\end{equation}
\end{lemma}
\begin{rmk}\label{rmk:315}
    According to the inversion invariance of Brownian motion, one can similarly derive the inward separation lemma. For brevity, we will not explicitly define the non-disconnection quality here, and will refer to it in subsequent proofs as the ``reverse separation lemma''.
\end{rmk}
A pair of Brownian paths will be called \textit{well-separated} when they satisfy the condition $\widetilde{\Delta}_n(\ol{\initialconfig})\ge {1}/{4}$. Analogous to the continuous case, define the rescaled non-disconnection quality for SRWs as
\begin{equation} \label{eq:disc-qua-1}
	\Delta_m(\ol{\initialconfig}):= e^{-m}\,\sup_{\delta>0} \big\{ 0\in \fr\big( \initialconfig^1\cup S^{1}[0,\tau_{m}]\cup \initialconfig^2 \cup S^{2}[0,\tau_{m}] \cup {B}(x^1_{m}, \delta e^m)\cup {B}(x^2_{m}, \delta e^m) \big) \big\},
\end{equation}
where $x_m^i$ is the terminal point of $\initialconfig^i\cup S^i[0,\tau_{m}]$. The discrete version of the separation lemma is as follows.
\begin{lemma}[Separation lemma, discrete case]\label{lem:rw-disc-sep}
	There exists a universal constant $c>0$ such that for any $0<l<m$ and $\ol\initialconfig\in \mathcal{Y}_l$,
	\begin{equation}
	    \Pb^{x_1,x_2} \big(\Delta_m(\ol{\initialconfig}) \ge 1/4\big)  \ge c \, \Pb^{x_1,x_2}\big(D_{m}(\ol {\initialconfig})\big) .
	\end{equation}
\end{lemma}
\begin{rmk}\label{rmk:329}
    Combining the coupling results of Brownian motion and simple random walk presented below, we can derive the reverse separation lemma for simple random walks from Remark \ref{rmk:315}.
\end{rmk}
With slight abuse of notation, we will similarly describe a pair of random walk paths as \textit{well-separated} when they satisfy the condition $\Delta_m(\ol{\initialconfig}) \ge 1/4$.
\paragraph{Coupling Techniques.} 
\textbf{(1) Skorokhod embedding.} 
To bridge discrete and continuous models, we employ a Skorokhod embedding scheme following \cite{L96b}. Let $X_1,\ldots,X_d$ be independent 1D Brownian motions. Define stopping times $\{\xi_n^j\}_{n\ge0}$ recursively by
\[
\xi_n^j = \inf\big\{t>\xi_{n-1}^j:|X^j(t)-X^j(\xi_{n-1}^j)|=1\big\}\ \ \mathrm{with}\ \ \xi_0^j=0.
\]
Let $Z_n = (Z_n^1,\ldots,Z_n^d)$ be a $d$-dimensional process independent of $\{X^j\}$, where $Z_0=0$ and $\Pb(Z_n-Z_{n-1} = e_j) = 1/d$ for unit vectors $e_j$, whose $j$-th exponent is equal to 1. This constructs coupled processes: 
\begin{itemize}
	\item $W(t) = \big(X^1(t),\ldots,X^d(t)\big)$ is a $d$-dimensional Brownian motion.
	\item $S(t) = \big(X^1(\xi^1(Z^1_{\left \lfloor  td\right \rfloor })),\ldots,X^d(\xi^d(Z^d_{\left \lfloor  td\right \rfloor }))\big)$ is a $d$-dimensional simple random walk.
\end{itemize}
The coupling ensures tight alignment between $W$ and $S$, quantified by the following lemma.
\begin{lemma}[Skorokhod embedding] 
    For any $\varepsilon>0$, there exists $u>0$ such that
	\begin{equation} \label{eq:se01}
		\Pb\big(\max_{0 \leq t \leq \tau_{n+1}\wedge T_{n+1}}|W(t)-S(t)|\geq e^{(1/2+\varepsilon )n}\big) = O(e^{-e^{un}}).
	\end{equation}
\end{lemma}
\noindent This exponential error decay is derived from moment estimates in \cite[Section 3]{L96b}. 

\textbf{(2) KMT coupling.} 
For sharper pathwise comparisons, we use the Koml\'os-Major-Tusn\'ady (KMT) coupling, which achieves near-optimal alignment between the simple random walk and Brownian motion (see \cite[Theorem 7.1.1]{LV10}).
\begin{lemma}[KMT coupling]
There exists a coupling $\Pb$ of planar random walk $S$ and Brownian motion $W$ such that for any $\delta>0$, there exists a constant $c_\delta>0$, such that for all $n\ge 1$,
\begin{equation}
\Pb\big(\max_{0\le t\le n}|W(t)-S(t)|\ge c_\delta\log n\big)\le c_\delta n^{-\delta}.
\end{equation} 
\end{lemma}
When analyzing exit times from the unit disk, we require a localized version of the KMT coupling. We state it as follows and refer the reader to \cite[Corollary 3.2]{KMG05} for a reference.
\begin{lemma}[Strong approximation]\label{lem:KMT1}
Under the KMT coupling, there exists some constant $K>0$, such that for all $n\ge 1$,
    \begin{equation}\label{eq:se11}
    \Pb (H^c)  =O( e^{-10n} )\ \mathrm{with}\ H:=\big\{\max_{0\le t\le \tau_{n}\vee T_{n}} |W(t)-S(t)|\le Kn \big\}.
    \end{equation}
\end{lemma}
\section{Continuum frontier disks and frontier Green's function}\label{sec3:Continuum frontier disks and frontier Green's function}
In this section, we formally introduce the concept of frontier disks and construct the induced measure $\widetilde\nu_s(V)$. Then, we establish the relationship between this induced measure $\widetilde\nu_s(V)$ and the occupation measure $\nu(V)$. Meanwhile, we define the frontier Green's functions $G_{\mathcal{D}}^{\fr}(z)$ and $G_{\mathcal{D}}^{\fr}(z,w)$, and apply the path decomposition technique to derive an estimate for $G_{\mathcal{D}}^{\fr}(z)$.

The coupling between the random walk and Brownian motion can be interpreted as a localized path perturbation. Such perturbations critically influence the geometric stability of points situated on the outer boundary of these paths. To tackle this, we propose a geometric construct that remains stable under minor fluctuations. 
\begin{defi}[Frontier disk in continuum] \label{def:non-disconnecting disk event continuous}
For $s>0$, a disk $B=\mathcal{D}_{-s}(z)\subseteq\mathcal{D}$ is called a  frontier disk of $W[0,T_0]$ if:
\begin{itemize}
\item $0\notin B$ and $W[0,T_0] \cap B \neq \emptyset$,
\item $\theta_1 \cup \theta_2$ does not disconnect $B$ and $\infty$,
\item $\omega \subseteq \mathcal{D}_{-2s/3}(z)$,
\end{itemize}
where $\theta_1 = W[0,\sigma]$ with $\sigma$ being the first hitting time of $B$, $\theta_2=W[\sigma',T_0]$ with $\sigma'$ being the last exiting time of $B$, and $\omega=W[\sigma,\sigma']$ is the intermediate part, forming the first-entry and last-exit decomposition of $W[0,T_0]$.
\end{defi}
Let ${\widetilde{K}_s}(z)$ denote the event that $\mathcal{D}_{-s}(z)$ is a frontier disk for $W[0,T_0]$. Recalling the two-arm disconnection exponent $\alpha = \xi(2)=2/3$, we define the frontier-disk measure as
\begin{equation}\label{eq:defLsz}
	\widetilde{\nu}_s(V):=\int_{V}L_s(z)dz,\ \ \mathrm{ where}\ \ L_s(z):=c's^{-1}e^{\alpha s}1_{{\widetilde{K}_s}(z)}.
\end{equation}
Here $c'>0$ is a normalizing constant, $e^{\alpha s}$ comes from the second non-disconnection condition, and the factor $s^{-1}$ compensates for geometric constraints on the intermediate part $\omega$. Define $J_s(z)$ via the near-frontier event:
\begin{equation}
J_s(z):=e^{\alpha s} \ind{z:\dist(z,\fr(W[0,T_0]))\leq e^{-s}}.
\end{equation}

The following results are the frontier-point analogues of Theorems 1.1-1.3 in \cite{HLLS22} which concerning Brownian cut points. Among them, Theorems \ref{prop:green function and frontier disk 1} and \ref{prop:green function and frontier disk 2} provide one-point and two-point estimates for frontier Green's function. Consequently, Theorem \ref{thm:418} provides an alternative approach to the existence of Minkowski content of Brownian frontier, without using SLE techniques. The proofs follow similarly as in their cut-point counterparts with minor technical modifications and are therefore omitted.

\begin{thm}[One-point frontier Green's function]\label{prop:green function and frontier disk 1}
For $s>0$ and $z\in\Dc$ with $dist(0,z,\partial \mathcal{D})\geq e^{-2s/3}$, the one-point frontier Green’s function exists as:
\begin{equation}\label{eq:427}
G_{\mathcal{D}}^{\fr}(z):=\lim_{s\to\infty}\Eb[L_s(z)].
\end{equation}
Furthermore, we have the asymptotic equivalence:
\begin{equation}
\label{eq:Lsz and green function}
\Eb[L_s(z)] \simeq G_{\mathcal{D}}^{\fr}(z),\ \ \mathrm{and}\ \ \Eb[J_s(z)] \simeq G_{\mathcal{D}}^{\fr}(z).
	\end{equation}
\end{thm}

Let $\Vc$ denote the collection of dyadic squares $V\subseteq\Dc$ of the form
\begin{equation}\label{eq:def of nice box}
	V=\Big[\frac{k_1}{2^n},\frac{k_1+1}{2^n}\Big]\times\Big[\frac{k_2}{2^n},\frac{k_2+1}{2^n}\Big],
\end{equation}
where $n,k_1,k_2\in \Zb$ satisfy $\min\{\dist(0,V), \dist(V,\partial\Dc)\}\geq 2\mathrm{diam}(V)$. We call elements of $\Vc$ \textit{``nice boxes''}. 
\begin{thm}[Two-point frontier Green's function] \label{prop:green function and frontier disk 2}
	For any $V \in \mathcal{V}$, $z,w \in V$ and $s>0$, the two-point frontier Green’s function exists as
    \begin{equation}
        G_{\mathcal{D}}^{\fr}(z,w):=\lim_{s\to\infty} \Eb[L_s(z)L_s(w)].
    \end{equation}
    Additionally, if $|z-w|\geq e^{-2s/3}$, the following asymptotic relations hold.
	\begin{equation}
		\Eb[L_s(z)L_s(w)]\simeq G_{\mathcal{D}}^{\fr}(z,w),\ \ \mathrm{and}\ \ 
	\Eb[L_s(z)J_s(w)]\simeq G_{\mathcal{D}}^{\fr}(z,w).
	\end{equation}
\end{thm}
\begin{thm}\label{thm:418}
	Let $V\subseteq \Dc$ be a bounded Borel set whose boundary $\partial V$ has zero $(2-\varepsilon) $-Minkowski content for some $\varepsilon >0$. Then the non-atomic measure $\nu(V)=\mathrm{Cont}_{\alpha}(\fr(W[0,T_0])\cap V)$ almost surely exists, with moments given by
	\begin{equation}\label{eq1:nu(V)}
		\Eb[\nu (V)]=\int_{V}G_{\mathcal{D}}^{\fr}(z)dz\ \ \mathrm{and}\ \ \Eb[\nu (V)^2]=\int_{V}G_{\mathcal{D}}^{\fr}(z,w)dzdw.
	\end{equation}
\end{thm}

With Theorems~\ref{prop:green function and frontier disk 1}-\ref{thm:418} established, we can obtain the $L^2$-distance estimate between the frontier-disk measure $\widetilde\nu_s(V)$ and the continuum occupation measure $\nu(V)$. The proof follows a similar approach to that of Theorem 6.14 in \cite{GLPS23} and we decide to omit it here.
\begin{thm}\label{thm:tilde_nu_s and nu}
	There exists $u>0$ such that if $V\subseteq \Vc$ and $\dist(0,V,\partial\Dc)\geq e^{-2s/3}$, then
	\begin{equation}
		\Eb[\widetilde\nu_s(V)-\nu(V)]^2  = O_V(e^{-us}).
	\end{equation}
\end{thm}

The following up-to-constants estimate of the frontier Green's function is established through a path decomposition technique.

\begin{lemma}\label{lem:518} For all $z\in \Dc$, let $d_z = \mathrm{dist}(0,z,\mathcal{D})$. Then we have
	\begin{equation}\label{eq:frontier green's function}
	G_{\mathcal{D}}^{\fr}(z)\asymp d_z^{1-\alpha}1_{|z|\ge1/2}+d_z^{1/4-\alpha}1_{|z|<1/2}=:a(z).
	\end{equation}
\end{lemma}
\begin{proof}
    We use the path-decomposition method. Let $(\initialconfig_1,\initialconfig_2)$ be a pair of non-disconnecting paths originating from the points $x_1,x_2$ on $\partial \mathcal{D}_{-s}(z)$ and stopped upon reaching $\partial D(z,d_z/2)$. The total mass associated with these configurations is $\asymp e^{-\alpha s} {d_z}^{-\alpha}$, since we observe that
    \begin{equation} \label{eq: total mass of non-disconnection excursions}
		\mu _{\partial \mathcal{D}_s,\partial \mathcal{D}_r }^{\mathcal{D}_r\setminus \overline{\mathcal{D}_s} }\otimes\mu _{\partial \mathcal{D}_s,\partial \mathcal{D}_r }^{\mathcal{D}_r\setminus \overline{\mathcal{D}_s} }[\initialconfig_1\mathrm{\,\,and\,\,} \initialconfig_2\mathrm{\,\,does\,\,not\,\,disconnect\,\,}\partial\Dc_s\mathrm{\,\,with\,\,}\infty ]\asymp e^{-\alpha(r-s)}.
	\end{equation}
    Let $\omega$ be the intermediate path connecting $x_1$ and $x_2$. The total mass that $\omega\subseteq \Dc_{-2s/3}(z)$ is $\asymp s$. Now, we analyze the remaining paths in two distinct cases.
    
    \noindent\textbf{Case 1. $|z|\ge1/2$.}
    \begin{itemize}
        \item Let $\eta_1$ be a path from $0$ to its first hitting of $\partial D(z,d_z/2)$ such that the distance between the endpoints of $\beta_1$ and $\eta_1$ is within $O(d_z)$. The total mass is $\asymp d_z$.
		\item Let $\eta_2$ be the path linking the endpoints of $\beta_1$ and $\eta_1$, which contributes a total mass of $\asymp 1$ by scaling invariance.
		\item Let $\eta_3$ be the path that starts from the endpoint of $\beta_2$ and stops upon reaching $\partial \mathcal{D}$. This path is confined within a well-chosen tube, contributing a total mass of $\asymp 1$.
    \end{itemize}
    \textbf{Case 2. $|z|< 1/2$.}
    \begin{itemize}
    \item Let $\eta_1$ be the path from $0$ to the endpoint of $\beta_1$, remaining within a designated tube, yielding a total mass of $\asymp 1$.
    \item Let $\eta_2$ be the path from the endpoint of $\beta_2$ to its first hitting of $\partial D(z, 2d_z)$, which also remains confined in a tube, contributing a mass of $\asymp 1$.
    \item Let $\eta_3$ be the path starting from the endpoint of $\eta_2$ to its first hitting of $\partial \mathcal{D}$, ensuring it does not disconnect $\partial D(z, 2d_z)$ from $\infty$. According to Lemma~\ref{one-arm disconnection exponent}, the total mass is $\asymp d_z^{1/4}$.
    \end{itemize}
    By multiplying the various mass factors together, we deduce that $$\Pb\big(\widetilde K_s(z)\big)\asymp e^{-\alpha s}s(d_z^{1-\alpha}1_{|z|\ge1/2}+d_z^{1/4-\alpha}1_{|z|<1/2}),$$
    implying that $G_{\mathcal{D}}^{\fr}(z)\asymp d_z^{1-\alpha}1_{|z|\ge1/2}+d_z^{1/4-\alpha}1_{|z|<1/2}$ by \eqref{eq:defLsz} and \eqref{eq:427}. 
\end{proof}

\section{Moment estimates for RW frontier points and frontier disks}\label{sec4:Moment estimates of frontier points and frontier disks}
In this section, we derive various moment bounds on the random walk frontier. In
Section \ref{sec4.1:Sharp moments estimates for RW frontier points}, we obtain sharp estimates of the moments of frontier-point events. 
In Section \ref{sec4.2:Inward coupling result}, we present a crucial inward coupling result. Section \ref{sec4.3:Discrete frontier disks and frontier points} introduces discrete frontier disks and establishes the relationship between discrete frontier-disk events and frontier-point events.

\subsection{Sharp moments estimates for RW frontier points}\label{sec4.1:Sharp moments estimates for RW frontier points}
This subsection focuses on establishing Theorem \ref{thm:one-point1}, which connects the discrete first and second moments of frontier points with the frontier Green's functions $G_{\mathcal{D}}^{\fr}(z)$ and $G_{\mathcal{D}}^{\fr}(z,w)$ defined in Section \ref{sec3:Continuum frontier disks and frontier Green's function}. We begin with key definitions.

For $z \in \mathcal{D} \setminus \{0\}$, define
\begin{equation*}
	d_z:=\mathrm{dist}(0,z,\partial \mathcal{D}) \ \ \mathrm{and}\ \  z_n:=\left\lfloor e^n z\right\rfloor,
\end{equation*}
where $\left\lfloor\cdot\right\rfloor$ denotes the floor function applied to each coordinate. Define the frontier-point event as
\begin{equation}
	\mathcal{A}_n(z): = \{z_n\ \mathrm{is\ on\ the\ frontier\ of\ the\ random\ walk\ path}\ S[0,\tau_n]\}.
\end{equation}
\begin{thm}\label{thm:one-point1}
There exists a universal constant $c_1>0$ such that:
\begin{itemize}
    \item[(i)] For all $z \in \mathcal{D}$ with $d_z\ge e^{-n/6}$,
	\begin{equation}
		c_1 e^{\alpha n}\Pb\big(\mathcal{A}_n(z)\big) \simeq G_{\mathcal{D}}^{\fr}(z).
	\end{equation}
    \item[(ii)] For all $V\in\mathcal{V}$, and $z,w \in V$ with $|z-w|\ge e^{-n/6}$,
	\begin{equation}
		c_1^2 e^{2\alpha n}\Pb\big(\mathcal{A}_n(z)\cap\mathcal{A}_n(w)\big)\simeq_V G_{\mathcal{D}}^{\fr}(z,w).
	\end{equation}
\end{itemize}
\end{thm}

To derive Theorem \ref{thm:one-point1}, we introduce ``nice configurations''. Given a curve $\initialconfig$ from $\partial\mathcal{B}_l$ to $\partial\mathcal{B}_m$ with $0<l<r\le m$, define its $r$-thickening as: 
\begin{equation*}
	\mathrm{THICK}_{r}(\initialconfig):=\{ x: \mathrm{dist}(x,\initialconfig[\tau_{r-1},\tau_r])\le e^{15r/16} \}\cup \initialconfig.
\end{equation*}
Moreover, we define the multi-scale thickening of $\initialconfig$ by 
\[
\mathrm{THICK}_{r_1,r_2}(\initialconfig):=\cup_{r_1\le r\le r_2}\mathrm{THICK}_{r}(\initialconfig),\  r_1<r_2.
\]
These notions of thickening can be naturally extended to multiple paths. 
We define the set of nice pairs of paths:
\[
\mathrm{NICE}_{m}:=\big\{ \ol\initialconfig\in \mathcal{Y}_m: 0\in \fr( \mathrm{THICK}_{\lfloor 29m/30\rfloor,m}(\ol\initialconfig) ) \big\}.
\]

The following proposition demonstrates that non-disconnecting SRW pairs exhibit niceness at intermediate scales with high probability. The full proof will be provided in Appendix \ref{appen:A}. We recall the definition of $\mathcal{Y}_l$ in \eqref{eq:Yl}.

\begin{prop}\label{prop:Dnf}
	There exist constants $c,u>0$ such that for any $\ol\initialconfig\in \mathcal{Y}_l$ and $n\ge m\ge 10 l$,
	\begin{equation}\label{eq:Dsh}
		 \Pb\big( D_{n}(\ol{\initialconfig})\,\cap\, \big\{ \ol\initialconfig\oplus\ol {S}[0,\tau_{m/2}]\notin \mathrm{NICE}_{m/2} \big\}  \big)  \le c e^{-u m} e^{-\alpha (n-l)}.
	\end{equation}
\end{prop}

    For $\ol\zeta\in \mathrm{NICE}_{m}$, we will view its thickening as initial configuration and consider the corresponding non-disconnection event for Brownian motions:
\begin{equation}
    \widetilde D_n\big(\mathrm{THICK}_{\lfloor 29m/30\rfloor,m}(\ol\zeta)\big) := \big\{0\in\fr\big(W_1[0,T_n]\cup W_2[0,T_n]\cup \mathrm{THICK}_{\lfloor 29m/30\rfloor,m}(\ol\zeta)\big)\big\},
\end{equation}
where $W_1,W_2$ are two independent Brownian motions started from the endpoints of $\zeta_1,\zeta_2$ respectively.
The following proposition then shows that the non-disconnection probability for SRWs with initial configuration $\ol\zeta$ is close to that for Brownian motions with thickening of $\ol\zeta$ as initial configuration instead.
For the proof, we refer the reader to Appendix \ref{appen:B}.
    \begin{prop}\label{prop:Dnf2}
		There exist constants $0<c,u<\infty$ such that for any $m\le n\le 10m$ and any $\ol\zeta\in\mathrm{NICE}_{m/2}$,  
		\begin{equation}\label{eq:779}
			\big{|} \Pb\big(  D_{n}(\ol\zeta) \big) - 
			\Pb\big( \widetilde D_n(\mathrm{THICK}_{\lfloor 29m/60\rfloor,m/2}(\ol\zeta)) \big) \big{|} \le c e^{-u m} e^{-\alpha (n-m/2)}.
		\end{equation}
	\end{prop}

Next, we assume that $z\in \Dc$ with $d_z=\dist(0,z,\partial\Dc)\ge e^{-n/6}$. 
Recall that $z_n=\left \lfloor e^nz \right \rfloor $ so that $\Bc_{5n/6}(z_n)\subset \Bc_n$. 
Given $\overline\zeta=(\zeta_1,\zeta_2)\in  \Yc_{n/6}(z_n)$ with endpoints $(x_1,x_2)\in \partial\Bc_{n/6}(z_n)\times\partial\Bc_{n/6}(z_n)$, define the set of pairs of discrete non-disconnecting paths with initial configuration $\ol\zeta$ that start from $0$ and $\partial\Bc_n$ respectively by
\begin{equation}\label{eq:Af'}
	A_{0,\partial\Bc_n}(\overline\zeta):=\big\{ (\lambda_1,\lambda_2)\in \Gamma^{\Bc_n}_{0,x_1}\times\Gamma^{\Bc_n}_{\partial\Bc_n,x_2}: z_n\in\fr\big((\zeta_1\cup\lambda_1)\cup(\zeta_2\cup\lambda_2)\big)\big\},
\end{equation} 
where $\Gamma^D_{x,y}$ denotes the collection of random walk paths from $x$ to $y$ within the region $D$. Define its continuous analogue by
\begin{equation}\label{eq:Af}
	\wt A_{0,\partial\Dc_n}(\overline\zeta):=\big\{
	(\initialconfig_1,\initialconfig_2)\in \wt\Gamma^{\Dc_n}_{0,x_1}\times\wt\Gamma^{\Dc_n}_{\partial\Dc_n,x_2}: z_n\in\fr\big((\wt\zeta_1\cup\initialconfig_1)\cup(\wt\zeta_2\cup\initialconfig_2)\big) \big\},
\end{equation}
where $\widetilde\zeta_i:=\mathrm{THICK}_{\left\lfloor 29n/180\right\rfloor,n/6}(\zeta_i)$ and $\wt\Gamma^D_{x,y}$ denotes the collection of continuous paths from $x$ to $y$ within $D$.

The following proposition establishes a key relationship between the total mass of $\wt A_{0,\partial\Dc_n}(\overline\zeta)$ and the frontier-point Green's function $G_{\mathcal{D}}^{\fr}(z)$. 
The proof of Proposition \ref{prop:536} is provided in Appendix \ref{appen:C}.
\begin{prop}\label{prop:536}
There exists $c_2>0$ such that for all $\ol\zeta\in \mathrm{NICE}_{n/6}(z_n)$ with endpoints $(x_1,x_2)$, we have
    \begin{equation}\label{eq:552(1)}
		\mu_{0,x_1}^{\Dc_n}\otimes\mu_{\partial\Dc_n,x_2}^{\Dc_n}\big(\wt A_{0,\partial\Dc_n}(\overline\zeta)\big) \simeq c_2 e^{-\alpha n/2} G_{\mathcal{D}}^{\fr}(z) 	\Pb\big( \widetilde D_{n/2}(\mathrm{THICK}_{\lfloor 29n/180\rfloor,n/6}(\ol\zeta))\big).
	\end{equation}
\end{prop}

Proposition \ref{prop:discrete and continuous 1} provides a sharp comparison between the probabilities of the discrete non-disconnection event and its continuous counterpart. The proof is in Appendix \ref{sec4.2:proof of prop3.12}.
\begin{prop} \label{prop:discrete and continuous 1}
	For all $\ol\zeta\in \mathrm{NICE}_{n/6}(z_n)$ with endpoints $(x_1,x_2)$, we have
	\begin{equation}\label{eq:560(1)}
		\nu_{0,x_1}^{\Bc_n}\otimes\nu_{\partial\Bc_n,x_2}^{\Bc_n}\big(A_{0,\partial\Bc_n}(\overline\zeta)\big) \simeq \mu_{0,x_1}^{\Dc_n}\otimes\mu_{\partial\Dc_n,x_2}^{\Dc_n}\big(\wt A_{0,\partial\Dc_n}(\overline\zeta)\big).
	\end{equation}
\end{prop}
We demonstrate the proof of Theorem~\ref{thm:one-point1} by leveraging the results established in Propositions~\ref{prop:Dnf}, \ref{prop:Dnf2}, \ref{prop:536}, and \ref{prop:discrete and continuous 1}.

\begin{proof}[Proof of Theorem~\ref{thm:one-point1}(i)]
	Let $p(\ol\zeta) = \Pb^{z_n,z_n}(\ol{S}[0,\tau_{n/6}] = \ol\zeta)$. Below, we use $(x_1,x_2)$ to denote the endpoints of $\ol\zeta$. Then,
	\begin{align}\label{eq:552}
		\Pb\big(\mathcal{A}_n(z)\big) &= \sum_{\ol\zeta\in\Yc_{n/6}(z_n)}\nu_{0,x_1}^{\Bc_n}\otimes\nu_{\partial\Bc_n,x_2}^{\Bc_n}\big(A_{0,\partial\Bc_n}(\overline\zeta)\big)p(\ol\zeta)\notag \\
		&\overset{\eqref{eq:Dsh}}{\simeq} \sum_{\ol\zeta\in\mathrm{NICE}_{n/6}(z_n)}\nu_{0,x_1}^{\Bc_n}\otimes\nu_{\partial\Bc_n,x_2}^{\Bc_n}\big(A_{0,\partial\Bc_n}(\overline\zeta)\big)p(\ol\zeta)\notag\\
		&\overset{\eqref{eq:560(1)}}{\simeq}\sum_{\ol\zeta\in\mathrm{NICE}_{n/6}(z_n)}\mu_{0,x_1}^{\Dc_n}\otimes\mu_{\partial\Dc_n,x_2}^{\Dc_n}\big(\wt A_{0,\partial\Dc_n}(\overline\zeta)\big)p(\ol\zeta)\notag \\ 
        &\overset{\eqref{eq:552(1)}}{\simeq} c_2 e^{-\alpha n/2} G_{\mathcal{D}}^{\fr}(z) \sum_{\ol\zeta\in\mathrm{NICE}_{n/6}(z_n)}\Pb\big( \widetilde D_{n/2}(\mathrm{THICK}_{\lfloor 29n/180\rfloor,n/6}(\ol\zeta))\big)
		p(\ol\zeta).
	\end{align}
	By Propositions~\ref{prop:Dnf} and \ref{prop:Dnf2},
	\begin{align*}
		&\sum_{\ol\zeta\in\mathrm{NICE}_{n/6}(z_n)}\Pb\big( \widetilde D_{n/2}(\mathrm{THICK}_{\lfloor 29n/180\rfloor,n/6}(\ol\zeta))\big)
		p(\ol\zeta) \notag\\ \overset{\eqref{eq:779}}{\simeq} &\sum_{\ol\zeta\in\mathrm{NICE}_{n/6}(z_n)}\Pb\big(  D_{n/2}(\ol\zeta)\big)
		p(\ol\zeta) \overset{\eqref{eq:Dsh}}{\simeq} \sum_{\ol\zeta\in\Yc_{n/6}(z_n)}\Pb\big(  D_{n/2}(\ol\zeta)\big)
		p(\ol\zeta)=\Pb( D_{n/2}) \overset{\eqref{sharp estimate for D_m}}{\simeq} qe^{-\alpha n /2}.
	\end{align*}
    Plugging the above estimate into \eqref{eq:552}, we get
	\[
	\Pb\big(\mathcal{A}_n(z)\big) \simeq c_2 G_{\mathcal{D}}^{\fr}(z) qe^{-\alpha n}.
	\]
    Picking $c_1 = 1/(c_2q)$, we conclude the proof.
\end{proof}

\subsection{Inward coupling result}\label{sec4.2:Inward coupling result}
We need an inward coupling result which we state without proof  below to establish the relation between the frontier-disk event and the frontier-point event in Section \ref{sec4.3:Discrete frontier disks and frontier points}.

Suppose $0<l\leq m\leq n-1\leq K-2$, and set $V_0=\Bc_K^c$. Suppose that $V_1$ and $V_2$ are two paths lying in $\Bc_n^c$ starting from $V_0$ and ending at $\partial \Bc_n$ such that $V_1 \cup V_2$ does not disconnect $\partial \Bc_n$ and $\infty$. For $j=1,2$, denote by $x_j$ the single endpoint of $V_j$. We refer to $(V_0,V_1,V_2,x_1,x_2)$ as an initial configuration. We often use $(V_0,U_1,U_2,y_1,y_2)$ to stand for another initial configuration.

For $j=1,2$, let $S^j$ be SRW started from $x_j$, $\Lambda_l^j = V_j \cup S^j[0,\tau_l^j]$ where $\tau_l^j=H_{\partial \Bc_l}(S^j)$. We define
\begin{align*}
	A_{l,n,K}^1&=A_{l,n,k}^1(V_0,V_1,V_2,x_1,x_2)=\{\Lambda_l^1\cup\Lambda_l^2\ \mathrm{does\ not\ disconnect}\ \partial \Bc_l\  \mathrm{and}\ \infty\},\\
	A_{l,n,K}^2&=\{\tau_l^1<\tau_K^1,\tau_l^2<\tau_K^2\},\mbox{ and }\\
	A_{l,n,K}&=A_{l,n,k}(V_0,V_1,V_2,x_1,x_2)=A_{l,n,K}^1\cap A_{l,n,K}^2.
\end{align*}
Let $\mu_{l,n,K}=\mu_{l,n,K}(V_0,V_1,V_2,x_1,x_2)$ be the probability measure induced by random walks $(S^1[0,\tau_l^1],S^2[0,\tau_l^2])$ conditioned on the event $A_{l,n,K}$. We say two pairs of paths $\ol\lambda = (\lambda^1,\lambda^2)$ and $\ol\lambda^* = (\lambda^{1,*},\lambda^{2,*})$ in $\Gamma_{\Bc_b^c,\partial \Bc_a}\times \Gamma_{\Bc_b^c,\partial \Bc_a}$ merge before hitting $\partial \Bc_b$ for the first time, if $\lambda^j(t+H_{\partial\Bc_b}(\lambda^j))=\lambda^{j,*}(t+H_{\partial\Bc_b}(\lambda^{j,*}))$ for all $t\geq 0$ and $j=1,2$. We write $\ol\lambda =_b \ol\lambda^*$ if they merge before hitting $\partial \Bc_b$.

\begin{thm}[Inward coupling for non-disconnecting walks]\label{thm:inward coupling}
	There exist $c,u>0$ such that for all $0<l\leq m\leq n-1\leq K-2$ and for any two initial configurations $(V_0,V_1,V_2,x_1,x_2)$ and $(V_0,U_1,U_2,y_1,y_2)$, we can define $\ol\lambda = (\lambda^1,\lambda^2)$ and $\ol\lambda^* = (\lambda^{1,*},\lambda^{2,*})$ on the same probability space $(\Omega, \Fc, P)$ such that $\ol\lambda$ has distribution $\mu_{l,n,K}(V_0,V_1,V_2,x_1,x_2)$ and $\ol\lambda^*$ has distribution $\mu_{l,n,K}(V_0,U_1,U_2,y_1,y_2)$. Moreover, under probability measure P, they are coupled as follows:
	\begin{equation}
	    P(\ol\lambda =_m \ol\lambda^*)\geq 1-ce^{-u(n-m)}.
	\end{equation}
\end{thm}
The proof is similar to that of Theorem 8.1 in \cite{GLPS23}, and hence omitted.

\subsection{Discrete frontier disks and frontier points} \label{sec4.3:Discrete frontier disks and frontier points}

In this subsection, we establish precise connections between the probabilities of frontier-disk events and frontier-point events. Our main results are Propositions~\ref{thm:one-point disk and point} and~\ref{thm:two-point disk and point}, which quantify these relationships. 
To begin with, we will define the frontier-disk event in the discrete setting, which is similar to Definition~\ref{def:non-disconnecting disk event continuous}.
\begin{defi}[Frontier disk in discrete]
	For $z \in \Dc\setminus 0$, the disk $\Bc_{3n/4}(z_n)$ is a frontier disk for the path $\lambda$ if the followings hold.
    \begin{itemize}
        \item $\lambda \cap \Bc_{3n/4}(z_n) \neq \emptyset$, 
        \item $\lambda_1 \cup \lambda_2$ does not disconnect $\Bc_{3n/4}(z_n)$ from $\infty$,
        \item and the intermediate segment $\omega \subseteq \Bc_{5n/6}(z_n)$,
    \end{itemize}
    where $\lambda = \lambda_1 \oplus \omega \oplus \lambda_2$ is decomposed according to its first and last visits to $\partial \Bc_{3n/4}(z_n)$.
\end{defi}
Let $K_{3n/4}(z)$ denote the event that $\Bc_{3n/4}(z_n)$ is a frontier disk for $\lambda$. The following two theorems, whose proofs are postponed to the end of this subsection, reveals the sharp relation between the frontier-disk event and the frontier-point event.
\begin{prop} \label{thm:one-point disk and point} 
	There exists a function $f(n)$ such that for all $z \in \mathcal{D}$ with $\mathrm{dist}(0,z,\partial \mathcal{D})\ge e^{-n/6}$, we have
	\begin{equation}
		\Pb\big(K_{3n/4}(z)\big)\simeq f(n)\Pb\big(\mathcal{A}_n(z)\big).
	\end{equation}
\end{prop}
Recalling the definition of $\Vc$ from \eqref{eq:def of nice box}, we have the following proposition.
\begin{prop} \label{thm:two-point disk and point}
	For any $V\in\mathcal{V}$ and $z,w\in V$ with $|z-w|\ge e^{-n/6}$, we have
	\begin{equation}
		\Pb\big(K_{3n/4}(z)\cap \mathcal{A}_n(w)\big)\simeq_V f(n)\Pb\big(\mathcal{A}_n(z)\cap \mathcal{A}_n(w)\big).
	\end{equation}
\end{prop}
Decompose the path $\lambda=S[0,\tau_n]$ into three segments by its first and last visits to $\Sc=\partial \Bc_{3n/4+1}(z_n)$. More precisely, we write $\lambda = \eta_1\oplus\omega_*\oplus[\eta_2]^R$, where
\begin{itemize}
	\item $\eta_1$ is the part of $\lambda$ from $0$ to its first hitting of $\Sc$,
	\item $\omega_*$ is the intermediate part between the first and last visits to $\Sc$,
	\item $\eta_2$ is the reversed path from $\tau_n$ to the last exit from $\Sc$.
\end{itemize}
Let $p_{0,n}^z$ denote the measure $\nu_{0,\Sc}^U \otimes \nu_{0,\Sc}^U$ restricted to the non-disconnecting event where $\eta_1\cup\eta_2$ does not disconnect $\Sc$ from $\infty$. Let $\mu_0$ be any reference measure $$\mu_{3n/4+1,19n/24,5n/6}(V_0,V_1,V_2,y_1,y_2)$$ defined in Section \ref{sec4.2:Inward coupling result}, and let $\mu_0^z$ be the image of $\mu_0$ under the translation $x\mapsto x+z_n$. We establish a coupling between $p_{0,n}^z$ and $\mu_0$ in the following lemma.
\begin{lemma}\label{lem:coupling of measures}
	There exists a coupling $\Pb$ of $\ol\eta$ and $\ol\eta'$ such that $\ol\eta\sim\widehat{p}_{0,n}^z$, $\ol\eta'\sim \mu_0^z$, and
	\begin{equation}
		\Pb(\ol\eta=_{37n/48}\ol\eta')\geq 1-ce^{-un},
	\end{equation}
	where $c,u>0$ are constants. Here $\ol\initialconfig=_{m}\ol\initialconfig'$ means that $\ol\initialconfig,\ol\initialconfig'\in \widetilde{\mathcal{Y}}$ merge before reaching $\partial \mathcal{D}_{-m}$.
\end{lemma}
\begin{proof}
	Let $\overline\eta=(\eta_1,\eta_2)\sim \widehat p^z_{0,n}$ and $\overline\eta'=(\eta'_1,\eta'_2)\sim \mu^z_0$. Decompose $\overline\eta=\overline \xi\oplus\overline\initialconfig$ where $\overline \xi=\overline\eta[0,H_{\partial\Bc_{19n/24}(z_n)}]$. By applying the strong Markov property, conditioned on $\overline\xi$ with endpoints $\overline y=(y_1,y_2)$, the distribution of $\overline\initialconfig$ corresponds to a pair of independent simple random walks starting from $\overline y$, stopped at $\Sc$, and conditioned on the event that they hit $\Sc$ before exiting $\Bc_n$ and $\eta_1\cup\eta_2$ does not disconnect $\Sc$ from $\infty$. By using Lemma~\ref{one-arm disconnection exponent}, there is a $u>0$ such that for $i=1,2$,
	\[
	\Pb\big(\initialconfig_i\cap \Bc_{5n/6}(z_n)^c \neq\emptyset \mid \ol\xi\,\big)=O(e^{-un}).
	\]
	Hence, upon further conditioning on  $\initialconfig_1$ and $\initialconfig_2$ remaining within $\Bc_{5n/6}(z_n)$, we have 
	\[
	\overline\initialconfig\sim \mu_{3n/4+1,19n/24,5n/6}(\Bc_{5n/6}(z_n)^c, \xi_1,\xi_2,y_1,y_2).
	\]
	Repeating this decomposition for $\overline\eta'$, we analogously find that conditioned on $\overline\xi'$ and $\overline\initialconfig'$ staying inside $\Bc_{5n/6}(z_n)$, 
	\[
	\overline\initialconfig'\sim \mu_{3n/4+1,19n/24,5n/6}(\Bc_{5n/6}(z_n)^c, \xi'_1,\xi'_2,y'_1,y'_2).
	\]
	Invoking the inward coupling (Theorem \ref{thm:inward coupling}) to $\ol\initialconfig$ and $\ol\initialconfig'$, we establish the desired result.
\end{proof}

Having established the coupling framework, we proceed by sampling $\ol\eta$ according to the measure $p_{0,n}^z$. Our analysis restricts consideration to the intermediate path segment $\omega_*$, which satisfies the frontier-disk event $K_{3n/4}(z)$ and frontier-point event $\Ac_{n}(z)$ respectively. We begin with the frontier-disk event. 

Let $\ol\eta=(\eta_1,\eta_2)\in\Gamma_{0,\Sc}^U \times \Gamma_{\partial \Bc_n,S}^U$ be a pair of non-disconnecting paths terminating at points $x_1,x_2$ on $\Sc$. Define $p^{\ol\eta}$ as the measure $\nu_{x_1,x_2}^{\partial \Bc_n}$ restricted to paths $\omega_*$ such that 
\begin{itemize}
	\item $\omega_*\cap \Bc_{3n/4}(z_n)\neq \emptyset$, decomposable as  $\omega_*=\intermediatepart_1\oplus\omega\oplus\intermediatepart_2$ according to its first and last visits to $\partial \Bc_{3n/4}(z_n)$;
	\item $(\eta_1\cup\intermediatepart_1)\cup(\eta_2\cup\intermediatepart_2)$ does not disconnect $\Bc_{3n/4}(z_n)$ and $\infty$, with $\omega\subseteq\Bc_{5n/6}(z_n)$.
\end{itemize}
We define the functionals
\begin{equation}
	\phi_1(\ol\eta):=\left \| p^{\ol\eta} \right \| , \,\,\mathrm{and}\,\,\phi_2(\ol\eta):=p^{\ol\eta}[1_{\intermediatepart_1,\intermediatepart_2\subseteq\Bc_{37n/48}(z_n)}].
\end{equation}
The probability of the frontier-disk event is expressed as
\begin{equation}\label{eq:716}
    \Pb\big(K_{3n/4}(z)\big)=p_{0,n}^{z}[\phi_1] = \left\| p_{0,n}^{z} \right\|  \widehat{p}_{0,n}^{z}[\phi_1].
\end{equation}
To estimate $\Pb(K_{3n/4}(z))$, we establish the following results related to the functionals $\phi_1$ and $\phi_2$.

\begin{lemma}\label{lem:estimate of phi_1}The functionals $\phi_1$ and $\phi_2$ satisfy
	\[
	\max_{\ol\eta}\phi_1(\ol\eta)\asymp \widehat{p}_{0,n}^{z}[\phi_2] \asymp n,
	\]
	where the maximum is taken over all $\ol\eta=(\eta_1,\eta_2) \in \Gamma_{0,\Sc}^U\times\Gamma_{\partial \Bc_n,\Sc}^U$ such that $\eta_1\cup\eta_2$ does not disconnect $\Sc$ from $\infty$ (in subsequent proofs, the maximum is always taken in this manner).
\end{lemma}
\begin{proof}
	From standard Green’s function estimates in $\Bc_n$, we derive $$\max_{\ol\eta}\phi_1(\ol\eta)\lesssim n.$$ For the lower bound, observe that $$\widehat{p}_{0,n}^{z}[\phi_2] \leq \widehat{p}_{0,n}^{z}[\phi_1] \leq \max_{\ol\eta}\phi_1(\ol\eta).$$ Thus, it suffices to prove $\widehat{p}_{0,n}^{z}[\phi_2] \gtrsim n$. By the reverse separation lemma (see Remarks \ref{rmk:315} and \ref{rmk:329}), $\ol\eta$ sampled from $\widehat{p}_{0,n}^{z}$ are well-separated (see Lemma \ref{lem:rw-disc-sep}) at $\Sc$ with positive probability. Conditional on this well-separation, we may attach $\intermediatepart_1$ and $\intermediatepart_2$ such that:
    \begin{itemize}
        \item $(\eta_1\cup\intermediatepart_1)\cup(\eta_2\cup\intermediatepart_2)$ does not disconnect $\Bc_{3n/4}(z_n)$ and $\infty$, 
        \item $\intermediatepart_1\cup\intermediatepart_2\subseteq\Bc_{37n/48}(z_n)$,
        \item and they are well-separated (see Lemma \ref{lem:rw-disc-sep}) at $\partial \Bc_{3n/4}(z_n)$.
    \end{itemize}
    For well-separated $\intermediatepart_1$ and $\intermediatepart_2$, the total mass of admissible $\omega$ is $\gtrsim n$.
\end{proof}

\begin{lemma} \label{compare}The functionals $\phi_1$ and $\phi_2$ satisfy
\begin{equation}\label{eq:738}
\widehat{p}_{0,n}^{z}[\phi_1] \simeq \widehat{p}_{0,n}^{z}[\phi_2].	    
\end{equation}
\end{lemma}
\begin{proof}
	It suffices to show \[
	\max_{\ol\eta}\big(\phi_1(\ol\eta)-\phi_2(\ol\eta)\big) = O(e^{-un})\max_{\ol\eta}\phi_1(\ol\eta).
	\]
	Note that
	\[
	\phi_1(\ol\eta)-\phi_2(\ol\eta) \le p^{\ol\eta}[1_{\intermediatepart_1\nsubseteq \Bc_{37n/48}(z_n)}]+p^{\ol\eta}[1_{\intermediatepart_2\nsubseteq \Bc_{37n/48}(z_n)}].
	\]
	If $\intermediatepart_1\nsubseteq \Bc_{37n/48}(z_n)$, decompose $\intermediatepart_1$ into a concatenation of a long path from $\Sc$ to $\Sc$ that reaches $\partial \Bc_{37n/48}(z_n)$ and another path from $\Sc$ to $\partial \Bc_{3n/4}(z_n)$. Using Lemma~\ref{one-arm disconnection exponent}, we get an extra cost that gives the $O(e^{-un})$ term. Summing these contributions yields the required exponential bound.
\end{proof}
\begin{lemma}\label{lem:total mass of K_3n/4}
The probability of the frontier-disk event satisfies the asymptotic equivalence:
	\begin{equation}
	    \Pb\big(K_{3n/4}(z)\big)\simeq \left \|{p}_{0,n}^{z}\right \| \mu_{0}^z[\phi_2].
	\end{equation}
\end{lemma}

\begin{proof}
	Observe that $\phi_2(\ol\eta)$ depends exclusively on the configuration of $\ol\eta$ within $\Bc_{37n/48}(z_n)$. By Lemma~\ref{lem:coupling of measures} (the coupling result), the discrepancy between the expectations under $\widehat{p}_{0,n}^z$ and $\mu_0^z$ is exponentially bounded:
	\begin{equation}\label{eq:761-0}
		|\widehat{p}_{0,n}^z[\phi_2]-\mu_0^z[\phi_2]|\lesssim e^{-un}\max_{\ol\eta}\phi_2(\ol\eta) \le e^{-un}\max_{\ol\eta}\phi_1(\ol\eta).
	\end{equation}
	Combining this bound with Lemma~\ref{lem:estimate of phi_1} gives that
	$\widehat{p}_{0,n}^z[\phi_2]\simeq\mu_0^z[\phi_2].$ Thus, we deduce that
	\[
	\Pb\big(K_{3n/4}(z)\big)\overset{\eqref{eq:716}}{=} \left\| p_{0,n}^{z} \right\|  \widehat{p}_{0,n}^{z}[\phi_1] \overset{\eqref{eq:738}}{\simeq}  \left\| p_{0,n}^{z} \right\|  \widehat{p}_{0,n}^{z}[\phi_2]\overset{\eqref{eq:761-0}}{\simeq} \left \|{p}_{0,n}^{z}\right \| \mu_{0}^z[\phi_2].
	\]This completes the proof. 
\end{proof}

We now consider admissible paths $w_*$ satisfying the frontier-point event $\Ac_n(z)$. Let $\ol\eta=(\eta_1,\eta_2)\in\Gamma_{0,\Sc}^U \times \Gamma_{\partial \Bc_n,S}^U$ denote a pair of non-disconnecting paths terminating at $x_1,x_2$ on $\Sc$. Define $q^{\ol\eta}$ as the measure $\nu_{x_1,x_2}^{\partial \Bc_n}$ restricted to paths $\omega_*$ satisfying: 
\begin{itemize}
	\item $z_n \in \fr(\eta_1\cup\omega_*\cup\eta_2);$
	\item $\omega_*$ decomposes as $\intermediatepart_1 \oplus \omega \oplus \intermediatepart_2$ according to its first and last visits to $\partial \Bc_{3n/4}(z_n)$.
\end{itemize}
Define the functionals:
\begin{equation}
\psi_1(\ol\eta):=\left\| q^{\ol\eta} \right\|, \,\,\mathrm{and}\,\,\psi_2(\ol\eta):= q^{\ol\eta}[1_{\intermediatepart_1,\intermediatepart_2\subseteq\Bc_{37n/48}(z_n)}].
\end{equation}
The probability of the frontier-point event is expressed as:
\[
\Pb\big(\Ac_{n}(z)\big) = p_{0,n}^z[\psi_1] = \left\| p_{0,n}^z \right\| \widehat{p}_{0,n}^z[\psi_1].
\]
\begin{lemma}The functionals $\psi_1$ and $\psi_2$ satisfy
	\[
	\max_{\ol\eta}\psi_1(\ol\eta)\asymp \widehat{p}_{0,n}^{z}[\psi_2] \asymp e^{-3\alpha n/4}.
	\]
	where the maximum is taken over all pairs $\ol\eta=(\eta_1,\eta_2) \in \Gamma_{0,\Sc}^U\times\Gamma_{\partial \Bc_n,\Sc}^U$ such that $\eta_1\cup\eta_2$ does not disconnect $\Sc$ from $\infty$.
\end{lemma}
\begin{proof}
	We construct admissible paths $\omega_*$ in the following manner:
	\begin{itemize}
		\item Let $(\xi_1,\xi_2)$ be a pair of non-disconnecting paths from $z_n$ to its first hitting of $\partial \Bc_{3n/4}(z_n)$ with endpoints $(y_1,y_2)$. By \eqref{sharp estimate for D_m}, the total mass of $(\xi_1,\xi_2)$ is $\asymp e^{-3\alpha n/4}$.
		\item Let $\intermediatepart_i$ be sampled from $\nu_{x_i,y_i}^{\Bc_n}$ such that the combined path $(\eta_1\cup\intermediatepart_1\cup\xi_1)\cup(\xi_2\cup\intermediatepart_2\cup\eta_2)$ does not disconnect $z_n$ from $\infty$. By confining $\zeta_i$ in a well-chosen tube, the total mass of $(\intermediatepart_1,\intermediatepart_2)$ is $O(1)$.
	\end{itemize}
	To establish the lower bound, it suffices to show $\widehat{p}_{0,n}^{z}[\psi_2] \gtrsim e^{-3\alpha n/4}$. By the separation lemma, the total mass of well-separated (see Lemma \ref{lem:rw-disc-sep}) $\ol\xi$ is $\gtrsim e^{-3\alpha n/4}$. Then the total mass of $\ol\intermediatepart$ we can attach is $\gtrsim 1$. Thus we conclude the proof.
\end{proof}

By using the same argument of Lemmas~\ref{compare} and \ref{lem:total mass of K_3n/4}, we are able to compare $\psi_1$ and $\psi_2$. Thus we get the following lemma. We omit the proof.
\begin{lemma}\label{lem:total mass of Ac}We have
	\begin{equation}
	    \widehat{p}_{0,n}^{z}[\psi_1] \simeq \widehat{p}_{0,n}^{z}[\psi_2]\ \ \mathrm{and}\ \ \Pb\big(\Ac_n(z)\big)\simeq\left \| p_{0,n}^z \right \| \mu_{0}^z[\psi_2].
	\end{equation}
\end{lemma}
We now turn to the proofs of Propositions \ref{thm:one-point disk and point} and \ref{thm:two-point disk and point}.
\begin{proof}[Proof of Proposition~\ref{thm:one-point disk and point}]
	Note that both $\mu_0^z[\phi_2]$ and $\mu_0^z[\psi_2]$ are translation invariant in $z$, thus
	\begin{equation}\label{eq:f(n)}
		f(n):=\mu_0^z[\phi_2] / \mu_0^z[\psi_2] \asymp ne^{3\alpha n/4}
	\end{equation}
	only depends on $n$. Then Proposition~\ref{thm:one-point disk and point} follows from Lemmas~\ref{lem:total mass of K_3n/4} and \ref{lem:total mass of Ac} directly.
\end{proof}
\begin{proof}[Proof of Proposition \ref{thm:two-point disk and point}]
The event $K_{3n/4}(z)\cap\Ac_n(w)$ can be decomposed into three sub-events:
	\begin{itemize}
		\item $E_1$: The frontier-disk event $K_{3n/4}(z)$ happens before $\Ac_n(w)$ and the SRW $S$ no longer returns within a vicinity of $z$ after making a frontier point at $w$.
		\item $E_2$: The frontier-disk event $K_{3n/4}(z)$ happens after $\Ac_n(w)$ and the SRW $S$ never enters the vicinity of $z$ before making a frontier point at $w$.
		\item $E_3$: The SRW $S$ visits the vicinity of $z$ at least twice and it makes a frontier point at $w$ between two visits and moreover forms a frontier disk around $z$.
	\end{itemize}
    We decompose the event $\Ac_n(z)\cap\Ac_n(w)$ similarly. Following the proof of Proposition \ref{thm:one-point disk and point}, one can match the first two sub-events $E_1,E_2$ with their corresponding sub-events decomposed from $\Ac_n(z)\cap\Ac_n(w)$. 
    We now analyze the third sub-event $E_3$. Due to an extra backtracking of the SRW path, there is an extra exponential penalty in probability by Lemma \ref{one-arm disconnection exponent}. Thus, the probability of $E_3$ is negligible compared with $\Pb(E_1)$ and $\Pb(E_2)$. We finish the proof. 
\end{proof}

\section{Convergence of occupation measure}\label{sec5:Convergence of occupation measure}
The goal of this section is to prove Theorem \ref{thm:main result 0}. In Section~\ref{sec5.1:Comparison between the continuum and discrete}, we employ the Skorokhod embedding theorem to demonstrate asymptotic equivalence between the first and second moments of continuous and discrete frontier-disk events. This equivalence establishes a rigorous correspondence between the discrete occupation measure $\nu_n(V)$ and the frontier-disk measure $\widetilde{\nu}_s(V)$. 
In Section~\ref{sec5.2:Proof of main theorem}, we finalize the proof of convergence of occupation measure. While the full argument parallels the methodology of \cite{GLPS23}, we only provide a proof sketch.
\subsection{Comparison between the continuum and discrete frontier disks}\label{sec5.1:Comparison between the continuum and discrete}

In this subsection, we construct a coupling between the random walk $S[0,\tau_{n+1}]$ and Brownian motion $W[0,T_{n+1}]$ on a common probability space using the Skorokhod embedding theorem (as specified in \eqref{eq:se01}) with parameter $\eps=1/8$. This coupling ensures
\begin{equation}\label{eq:se0}
	\Pb (H^c)  =O( e^{-10n} )\    \mathrm{with} \ 
	H:=\big\{\max_{0\le t\le \tau_{n+1}\vee T_{n+1}} |S_t-W_t |\le e^{5n/8}\big\}.
\end{equation}
Under this coupling, the discrete frontier-disk event $K_{3n/4}(z)$ can be approximated by its continuous counterpart $\widetilde{K}_{3n/4}^{(n)}(z)$, where $\widetilde{K}_{3n/4}^{(n)}(z)$ corresponds to the event $\widetilde{K}_{n/4}(z)$ scaled spatially by $e^n$. By Lemma \ref{lem:518} and the scaling invariance property of Brownian motion, we obtain the following estimate.

\begin{coro}
\label{lem:asymp}
	For all $z\in \Dc$ with $\dist(0,z,\partial \Dc)\geq e^{-n/6}$, we have
	\begin{equation}
		\Pb\big(K_{3n/4}(z)\big) \asymp \Pb\big(\widetilde{K}_{3n/4}^{(n)}(z)\big) \asymp a(z)e^{-\alpha n/4}n.
	\end{equation}
\end{coro}

The next lemma shows that the probability of the discrete frontier-disk event $K_{3n/4}(z)$ is asymptotically equivalent to the probability of the continuous counterpart $\widetilde{K}_{3n/4}^{(n)}(z)$. The proof strategy involves showing that the symmetric difference between the two events has negligible probability.
\begin{lemma} \label{lem:compare K with tildeK}
	For all $z\in \Dc$ with $\dist(0,z,\partial \Dc)\geq e^{-n/6}$, we have
	\begin{equation}
	    \Pb\big(K_{3n/4}(z)\big) \simeq \Pb\big(\widetilde{K}_{3n/4}^{(n)}(z)\big).
	\end{equation}
\end{lemma}
\begin{proof}
	By Corollary~\ref{lem:asymp}, it suffices to bound the symmetric difference under the coupling \eqref{eq:se0}, i.e. for some $u>0$, we have
	\[
	\Pb\big(K_{3n/4}(z) \bigtriangleup \widetilde{K}_{3n/4}^{(n)}(z)\big) \lesssim a(z)e^{-\alpha n/4}e^{-un}.
	\]
	By symmetry, we only analyze $E_0=K_{3n/4}(z) \setminus \widetilde{K}_{3n/4}^{(n)}(z)$. Let $B$ (resp.\ $B_{\pm}$) denote discrete disks of radius $e^{3n/4}$ (resp. $e^{3n/4}\pm 2e^{5n/8}$) around $z_n$. Decompose the path $\lambda=S[0,\tau_n]$ as $\lambda_1\oplus\omega\oplus\lambda_2^R$ (resp. $\lambda_1^{\pm}\oplus\omega^{\pm}\oplus[\lambda_2^{\pm}]^R$) according to its first and last visits to $B$ (resp. $B_{\pm}$). Recall the event $H$ from \eqref{eq:se0}. On the event $H\cap E_0$, at least one of the following four events occurs:
	\begin{itemize}
		\item $F_1$: $\lambda[\tau_n, \tau_{\partial B(e^n+e^{5n/8})}]\cap \Bc_{11n/16}(\lambda(\tau_n))^c\neq\emptyset$.
		\item $F_2$:  $\omega\cap B_{-}=\emptyset$.
		\item $F_3$: $\lambda_1\cup\lambda_2$ does not disconnect $z_n$ from $\infty$ but the $e^{5n/8}$-sausage $\sausage(\lambda_1\cup\lambda_2,e^{5n/8})$ does.
		\item $F_4$: $\omega^+\nsubseteq B(z_n,e^{5n/6}-2e^{5n/8})$.
	\end{itemize}
	Then we have $\Pb(H\cap E_0) \leq \sum_{i=1}^{4} \Pb(E_0\cap F_i) $. It suffices to show that
	\begin{equation}\label{eq:F_i}
		\Pb(E_0\cap F_i) \lesssim a(z)e^{-\alpha n/4}e^{-un},\,i=1,2,3,4.
	\end{equation}
	By gambler’s ruin estimate (see e.g.\ \cite[Theorem 3.18]{MP10}), the above estimate holds for $F_1$. Standard Green’s function estimates imply \eqref{eq:F_i} holds for $F_2$. Analogous to Proposition \ref{prop:Dnf}, we know that \eqref{eq:F_i} also holds for $F_3$. It remains to deal with $F_4$. For $x,y\in\partial B_+$, let $\omega^+$ be sampled according to $\nu^{\Bc_n}_{x,y}$. Let $\intermediatepart_1\oplus\omega\oplus\intermediatepart_2^R$ be the decomposition of $\omega^+$ according to its first and last visits to $B$. 
	By the gambler's ruin estimate again, there exists  $u>0$ such that for all $x$ and $y$,
	\begin{equation*}\label{eq:G}
		\nu^{\Bc_n}_{x,y}\{ \omega\subset\Bc_{5n/6}(z_n), \omega^+\nsubseteq B(z_n,e^{5n/6}-2e^{5n/8}) \}=O(e^{-un}).
	\end{equation*}
	This implies \eqref{eq:F_i} holds for $F_4$.
\end{proof}
Lemma \ref{lem:total mass of second moment 1} establishes up-to-constants probability estimates for cross terms involving discrete frontier-disk event $K_{3n/4}(z)$ and frontier-point event $\Ac_n(w)$. Lemma \ref{lem:total mass of second moment 2} derives analogous up-to-constants bounds for the continuum frontier-disk event $\widetilde{K}_{{3n/4}}^{(n)}(z)$ and the frontier-point event $\Ac_n(w)$ through the Skorokhod embedding theorem with a decoupling technique.
\begin{lemma}\label{lem:total mass of second moment 1}
	For all $V\in\Vc$, $z,w\in V$, with $|z-w|\ge e^{-n/6}$, we have
	\begin{equation}	
		\Pb\big(K_{3n/4}(z)\cap\Ac_n(w)\big) \asymp_V |z-w|^{-\alpha}e^{-5\alpha n/4}n.
	\end{equation}
\end{lemma}
\begin{proof}
	Let $d_V=\dist(0,V,\partial \Dc)$, then $d_V\ge e^{-n/6}$ from the definition of a nice box. We decompose the event $K_{3n/4}(z)\cap \mathcal{A}_n(w)$, using the same method employed in the proof of Proposition \ref{thm:two-point disk and point}. 
	To get the probability of these sub-events $E_1,E_2,E_3$, we still use the path-decomposition method. We focus on $E_1$ first and decompose $\lambda=S[0,\tau_n]$ as follows.
	\begin{itemize}
		\item Let $(\initialconfig_1,\initialconfig_2)$ be a pair of non-disconnecting paths sampled from the boundary-to-boundary measure in $B(z_n,|z-w|e^n/4)\setminus B(z_n,e^{3n/4})$, i.e., the excursion measure $$\nu_{\partial B(z_n,e^{3n/4}),\partial B(z_n,|z-w|e^n/4)}^{B(z_n,|z-w|e^n/4)\setminus\ol{B}(z_n,e^{3n/4})}\otimes\nu_{\partial B(z_n,e^{3n/4}),\partial B(z_n,|z-w|e^n/4)}^{B(z_n,|z-w|e^n/4)\setminus\ol{B}(z_n,e^{3n/4})},$$ such that $\initialconfig_1\cup\initialconfig_2$ does not disconnect $\partial \Bc_{3n/4}(z_n)$ and $\infty$. The total mass of $(\initialconfig_1,\initialconfig_2)$ is $\asymp (\frac{|z-w|e^n/4}{e^{3n/4}})^{-\alpha}$ (see Lemma \ref{thm:total mass} for the continuous case).
		\item Let $(\initialconfig_3,\initialconfig_4)$ be a pair of non-disconnecting paths from $w_n$ to $\partial B(w_n,|z-w|e^n/4)$. By \eqref{sharp estimate for D_m}, the total mass is $\asymp (|z-w|e^n/4))^{-\alpha}$.
		\item Let $\xi$ be the path connecting the endpoints of $\initialconfig_2$ and $\initialconfig_3$ with total mass $\asymp 1$.
		\item Let $\omega$ be the path connecting the endpoints of $\initialconfig_1$ and $\initialconfig_2$ such that $\omega\subseteq \Bc_{5n/6}(z_n)$. By a standard estimate of Green's function, the total mass is $\asymp n$.
		\item Let $(\eta_1,\eta_2)$ be a pair of non-disconnecting paths from the endpoints of $\initialconfig_1$ and $\initialconfig_4$ respectively to $\partial B((z_n+w_n)/{2},d_Ve^n/2)$. By \eqref{eq:non-disconnection probability}, the total mass of $(\eta_1,\eta_2)$ is $\asymp {(\frac{d_Ve^n}{|z-w|e^n})}^{-\alpha}$.
		\item Let $(\lambda_1,\lambda_2)$ be sampled from the path measure $\nu_{0,x_1}^{\Bc_n}\otimes\nu_{\partial \Bc_n,x_2}^{\Bc_n}$ such that $\lambda_1$ and $\lambda_2$ does not disconnect $\partial B((z_n+w_n)/{2},d_Ve^n/2)$ and $\infty$, where $x_i$ is the endpoint of $\eta_i$. By scaling invariance, the total mass of $(\lambda_1,\lambda_2)$ is $\asymp c(V)$.
	\end{itemize}
	Multiplying all of the total masses above, we get the probability of $E_1$. 
	\begin{align*}
		\Pb(E_1) &\asymp_V \Big(\frac{|z-w|e^n/4}{e^{3n/4}}\Big)^{-\alpha} \times (|z-w|e^n/4)^{-\alpha} \times 1 \times n \times {\Big(\frac{d_Ve^n}{|z-w|e^n}\Big)}^{-\alpha}\\ &\asymp_V |z-w|^{-\alpha}e^{-5\alpha n/4}n.
	\end{align*}
	By symmetry, we also know that the probability of $E_2$ has the same order, i.e.,
	\[
	\Pb(E_2) \asymp_V |z-w|^{-\alpha}e^{-5\alpha n/4}n.
	\]
	Following the same reasoning as in Proposition \ref{thm:two-point disk and point} (an additional backtracking), the probability of event $E_3$ becomes negligible relative to $\Pb(E_1)$ and $\Pb(E_2)$. This completes the proof.
\end{proof}

\begin{lemma} \label{lem:total mass of second moment 2}
	Under the coupling \eqref{eq:se0}, for all $V\in\Vc$, $z,w\in V$, with $|z-w|\ge e^{-n/6}$, we have
	\begin{equation}
		\Pb\big(\widetilde{K}_{{3n/4}}^{(n)}(z)\cap\Ac_n(w)\big) \asymp_V |z-w|^{-\alpha}e^{-5\alpha n/4}n.
	\end{equation}
\end{lemma}
\begin{proof}
	In this case, we need Lemma 3.7 of \cite{GLPS23}, a certain type of Strong Markov Property, to decouple BM and SRW. We also decompose the event $\widetilde{K}_{{3n/4}}^{(n)}(z)\cap\Ac_n(w)$ into three sub-events:
	\begin{itemize}
		\item $\widetilde{E}_1$: The frontier-disk event $\widetilde{K}_{{3n/4}}^{(n)}(z)$ happens before $S$ creates a frontier point at $w_n$, and $W$ never returns the vicinity of $z_n$ after $S$ creates a frontier point at $w_n$.
		\item $\widetilde{E}_2$: The frontier-disk event $\widetilde{K}_{{3n/4}}^{(n)}(z)$ happens after $S$ creates a frontier point at $w_n$, and $W$ never enters the vicinity of $z_n$ before $S$ creates a frontier point at $w_n$.
		\item $\widetilde{E}_3$: The BM $W$ enters the vicinity of $z_n$ at least twice, and the SRW $S$ makes a frontier point at $w_n$ between two visits of $\Dc_{3n/4}(z_n)$.
	\end{itemize}
	Due to the same reason in the proof of Proposition \ref{thm:two-point disk and point}, we only need to concentrate on $\widetilde{E}_1$, since $\widetilde{E}_2$ is symmetric and $\widetilde{E}_3$ is negligible. We decompose $W[0,T_n]$ in the following way.
	\begin{itemize}
		\item Let $\initialconfig_1$ be $W$ started from 0 until its first visit of $\partial \Dc_{3n/4}(z_n)$.
		\item Let $\initialconfig_2$ be $W$ started from the endpoint of $\initialconfig_1$ and stopped until hitting $\partial D(z_n,|z-w|e^n/4)$, and we denote this hitting time by $\sigma$.
		\item Let $\initialconfig_3$ be $W$ started from the endpoint of $\initialconfig_2$ and stopped until hitting $\partial D(w_n,|z-w|e^n/4)$ after its last visit of $\Dc_{5n/8}(w_n)$.
		\item Let $\initialconfig_4$ be $W$ started from the endpoint of $\initialconfig_3$ and stopped until hitting $\partial \Dc_n$.
	\end{itemize}
	Using Lemma 3.7 of \cite{GLPS23} with $\eps={1}/{8}$ and $b={5}/{7}>{1}/{2}+\eps$, there exists an event $\Upsilon_{w,n}$ with
	\[
	\Pb\big(\Upsilon_{w,n}^c \cap \Ac_{n}(w)\big) \lesssim e^{-10n}.
	\]
	such that on $\Upsilon_{w,n} \cap \Ac_{n}(w)$, we have
	\begin{equation} \label{eq:decouple}
		\initialconfig_1 \cup \initialconfig_2 \cup \initialconfig_4 \,\,\mathrm{and}\,\,S[\iota_1,\iota_2]\,\,\mathrm{are\ conditionally\ independent\ given}\,\,\initialconfig_3,
	\end{equation}
	where
	\[
	\iota_1=\inf\{t>\sigma:S(t)\in\partial B(w_n,|z-w|e^n/8) \},
	\]
	and
	\[
	\iota_2=\inf\{t>\tau_{w_n}:S(t)\in\partial B(w_n,|z-w|e^n/8) \}.
	\]
	If $w_n$ is a frontier point of $S$, then it is also a frontier point of $S[\iota_1,\iota_2]$, which satisfies
	\begin{equation} \label{eq:total mass 1}
		\Pb\big(w_n\in \fr(S[\iota_1,\iota_2])\big)\asymp(|z-w|e^n)^{-\alpha}.
	\end{equation}
	Now, it suffices to compute the total mass of $(\initialconfig_1,\initialconfig_2,\initialconfig_4)$ that satisfies the frontier-disk event $\widetilde{K}_{{3n/4}}^{(n)}(z)$ as follows.
	\begin{itemize}
		\item Let $\initialconfig_1^1$ be the part of $\initialconfig_1$ started from 0 and stopped until its last visit of $\partial D(z_n,|z-w|e^n/4)$. By \eqref{eq:non-disconnection probability}, the total mass of non-disconnecting paths $(\initialconfig_1^1,\initialconfig_4)$ is $\asymp_V (\frac{d_Ve^n}{|z-w|e^n})^{-\alpha}$.
		\item Let $\initialconfig_1^2=\initialconfig_1\setminus\initialconfig_1^1$ and $\initialconfig_2^2=\initialconfig_2\setminus\initialconfig_2^1$ where $\initialconfig_2^1$ is the part of $\initialconfig_2$ started from the endpoint of $\initialconfig_1$ and stopped until its last visit of $\partial \Dc_{3n/4}(z_n)$. By \eqref{eq:non-disconnection probability}, the total mass of non-disconnecting paths  $(\initialconfig_1^2,\initialconfig_2^2)$ is $\asymp (\frac{|z-w|e^n}{e^{3n/4}})^{-\alpha}$.
		\item By a standard estimate of Green's function, the total mass of $\initialconfig_2^1\subseteq\Dc_{5n/6}(z_n)$ is $\asymp n$.
	\end{itemize}
	The multiplication of the above factors implies that the total mass of $(\initialconfig_1,\initialconfig_2,\initialconfig_4)$ that satisfies the frontier-disk event $\widetilde{K}_{{3n/4}}^{(n)}(z)$ is
	\begin{equation} \label{eq:total mass 2}
		\asymp_V \Big(\frac{d_Ve^n}{|z-w|e^n}\Big)^{-\alpha} \times \Big(\frac{|z-w|e^n}{e^{3n/4}}\Big)^{-\alpha} \times n \asymp_V e^{-\alpha n/4}n.
	\end{equation}
	According to \eqref{eq:decouple}, we can multiply \eqref{eq:total mass 1} by \eqref{eq:total mass 2} to conclude the proof.
\end{proof}
The following lemma strengthens these second moments results by proving an asymptotic equivalence between the joint probabilities of discrete frontier-disk and frontier-point events and their continuous analogues, demonstrating that their second moments converge to the same order with an exponentially small error.
\begin{lemma}\label{lem:second moment discrete and continuum}
	Under the coupling \eqref{eq:se0}, for all $V\in\Vc$, $z,w\in V$, with $|z-w|\ge e^{-n/6}$, we have
	\begin{equation}
		\Pb\big(K_{3n/4}(z)\cap\Ac_n(w)\big)\simeq_V \Pb\big(\widetilde{K}_{{3n/4}}^{(n)}(z)\cap\Ac_n(w)\big).
	\end{equation}
\end{lemma}
\begin{proof}
	Let $$G_1=\big(K_{3n/4}(z)\cap\Ac_n(w)\big)\setminus\big(\widetilde{K}_{{3n/4}}^{(n)}(z)\cap\Ac_n(w)\big),$$ and $$G_2=\big(\widetilde{K}_{{3n/4}}^{(n)}(z)\cap\Ac_n(w)\big)\setminus\big(K_{3n/4}(z)\cap\Ac_n(w)\big).$$
	
	By Lemmas~\ref{lem:total mass of second moment 1} and \ref{lem:total mass of second moment 2}, we only need to show that for $i=1,2$, we have
	\begin{equation}\label{E_i small}
		\Pb(G_i)\lesssim_V |z-w|^{-\alpha}e^{-5\alpha n/4}e^{-un}.
	\end{equation}
	Let us deal with $G_2$ for illustration. The notation in the following items are introduced in the proof of Lemma~\ref{lem:total mass of second moment 2}. We will only sketch the proof, since it is similar to that of Lemma~\ref{lem:compare K with tildeK}.
	Under the Skorokhod embedding \eqref{eq:se0}, on the event $G_2$, one of the following four events will happen:
	\begin{itemize}
		\item (a): $\initialconfig_2^1\cap \partial \Dc_{5n/6}(z_n) = \emptyset$, but the $e^{5n/8}$-sausage $\sausage(\initialconfig_2^1,e^{5n/8})$ intersects $\partial \Dc_{5n/6}(z_n)$.
		\item (b): $\initialconfig_1^1\cup\initialconfig_4$ does not disconnect $z_n$ from $\infty$ but the $e^{5n/8}$-sausage $\sausage(\initialconfig_1^1\cup\initialconfig_4,e^{5n/8})$ does.
		\item (c): $\initialconfig_1^2\cup\initialconfig_2^2$ does not disconnect $z_n$ from $\infty$ but the $e^{5n/8}$-sausage $\sausage(\initialconfig_1^2\cup\initialconfig_2^2,e^{5n/8})$ does.
		\item (d): $\initialconfig_1\cup\initialconfig_3$ does not disconnect $z_n$ from $\infty$ but the $e^{5n/8}$-sausage $\sausage(\initialconfig_1\cup\initialconfig_3,e^{5n/8})$ does.
	\end{itemize}
	Now, using the same strategy as that of Proposition~\ref{lem:compare K with tildeK}, we can always obtain an extra cost $O(e^{-un})$ from (a), (b), (c) and (d) respectively. Thus, we conclude that \eqref{E_i small} holds for $G_2$. The event $G_1$ can be analyzed in a similar way. We thus finish the proof of this proposition.
\end{proof}

\subsection{Proof of Theorem \ref{thm:main result 0}}\label{sec5.2:Proof of main theorem}
According to the Portmanteau theorem, to prove Theorem \ref{thm:main result 0}, it suffices to show that under the Skorokhod embedding \eqref{eq:se0} (assumed below), for any bounded continuous function $g:\ol\Dc\to\Rb$, we have
	\begin{equation}\label{eq:vg}
		\lim_{n \to \infty} \Eb[\nu_n(g)-\nu(g)]^2=0.
	\end{equation}
To prove \eqref{eq:vg}, we first show the convergence of occupation measures on boxes $V\in\Vc$. Following the same lines as in the proof of Proposition 10.1 of \cite{GLPS23}, and using Proposition \ref{thm:two-point disk and point}, Lemma \ref{lem:second moment discrete and continuum}, equation \eqref{eq:f(n)} and Theorem \ref{thm:one-point1}(ii) as inputs now, we can deduce that for all $V\in \Vc$ with $\dist(0,V,\partial \Dc) \geq e^{-n/6}$,
	\begin{equation*}
		\Eb[\nu_n(V) - \widetilde\nu_{n/4}(V)]^2 = O_V(e^{-un}).
	\end{equation*}
This combined with Theorem \ref{thm:tilde_nu_s and nu} yields that for any nice box $V$ (recall \eqref{eq:def of nice box} for definition) in the unit disk $\mathcal{D}$,
	\begin{equation}\label{eq:convergence of indicator}
		\Eb[\nu_{n}(V)-\nu(V)]^{2}= O_V(e^{-un}).
	\end{equation}
Finally, we conclude \eqref{eq:vg} by approximating $g$ by step functions (similar to the proof of Theorem 1.2 in \cite{GLPS23}). We finish the proof of Theorem \ref{thm:main result 0} and omit further details.

\section{Convergence under natural parametrization}\label{sec6:Convergence under natural parametrization}
In this section, we are going to prove Theorem \ref{thm:convergence under np}, the convergence of $\{\gamma_n\}$ under natural parametrization, where $\gamma_n$ is the random walk frontier defined in Section \ref{sec1.1:Convergence under natural parametrization 1}. The proof consists of two parts. First, we prove in Section \ref{sec6.1:convergence under rp} that $\gamma_n$ converges to $\widetilde\gamma$ under reparametrization metric, which is done by first proving tightness and then uniqueness, making use of the result of convergence under Hausdorff metric. Then, we finish the proof in Section \ref{sec6.2:Proof of Theorem 1.1} by combining the convergence under reparametrization and convergence of occupation measure. 
\subsection{Convergence under reparametrization}\label{sec6.1:convergence under rp}
Using the KMT coupling from Lemma \ref{lem:KMT1} and scaling by $e^{-n}$, we can couple the random walk $S_n[0,\tau_{0}]$ and the Brownian motion $W[0,T_{0}]$ on a common probability space such that:
\begin{equation}\label{eq:se1}
	\Pb (H^c)  =O(e^{-10n})\ \    \mathrm{with}\ \ 
	H:=\big\{\max_{0\le t\le \tau_{0}\vee T_{0}} |W(t)-S_n(t) |\le K n e^{-n}\big\}.
\end{equation}
Recall from Section \ref{sec1.1:Convergence under natural parametrization 1} that $\Pc = \{\gamma:[0,t_\gamma]\to\Rb^2\}$ denotes planar curves. 
The metric modulo reparametrization, $d_{Re}$ is formally defined by
\begin{equation}
    d_{Re}(\gamma,\gamma') = \inf_{\alpha,\alpha'}\big\{  \sup_{s\in [0,1]} \left| \gamma(\alpha(s))-\gamma'(\alpha'(s)) \right|\big\},
\end{equation}
with infimum over increasing continuous bijections $\alpha:[0,1]\to [0,t_{\gamma}]$ and $\alpha':[0,1]\to [0,t_{\gamma'}]$. The main result of this section is the following theorem.
\begin{thm}\label{thm:main result 3}
Almost surely, $\gamma_n$ converges to $\widetilde \gamma$ under $d_{Re}$.
\end{thm}
To establish Theorem \ref{thm:main result 3}, we will first demonstrate the tightness of the sequence $\{\gamma_n\}$. This property ensures that every subsequence admits a further convergent subsequence. Then, we will verify that all such subsequential limits coincide, thereby confirming their uniqueness.
\subsubsection{Tightness}
In this subsection, we will show the tightness of $\{\gamma_n\}$ under $d_{Re}$, using a result of Aizenman Burchard \cite{AB99}. We first recall Theorem 1.2 from \cite{AB99}.
\begin{lemma}[\cite{AB99},\ Theorem 1.2]\label{lem:AB99}
    Let $\{\gamma_n\}$ be a system of random curves, where the step size of the polygonal paths $\gamma_n$ is $1/n$. Assume that $\{\gamma_n\}$ satisfies the following hypothesis\footnote{This power bound on the probability of multiple crossings is called \textit{Hypothesis H1} in \cite{AB99}.}: 
    for all integers $k \ge1$, there exists $\lambda(k) \to \infty$ as $k \to \infty$, such that for any annulus $\annulus(x,r,R)$ with radii $1/n\le r<R\le 1$,  
    \begin{equation}\label{eq:H1}
	\Pb \big( \annulus(x,r,R) \text{ is traversed by $\gamma_n$ at least $k$ times}
	\big) \lesssim_k \left({r}/{R}\right)^{\lambda(k)}.
\end{equation}Then, $\{\gamma_n\}$ is tight with respect to $d_{Re}$.
\end{lemma}
We now prove the following theorem on the tightness of $\{\gamma_n\}$.
\begin{thm}
	The sequence $\{\gamma_n\}$ is tight under $d_{Re}$.
\end{thm}
\begin{proof}
    We first verify \eqref{eq:H1} for $\{\gamma_n\}$. 
Actually, this can be derived from Lemma~\ref{one-arm disconnection exponent} and the strong Markov property of Brownian motion, with $\lambda(k)= k/4$ (indeed, one can choose $\lambda(k)$ to be the disconnection exponent $\xi(k)$, but we do not need such a stronger estimate here). Then, the tightness follows from Lemma \ref{lem:AB99} immediately.
\end{proof}

\subsubsection{Uniqueness}
In this subsection, we will show the uniqueness of subsequential limits of $\{\gamma_n\}$ under $d_{Re}$ by analyzing the number of {\it bad} disks (see \eqref{eq1130} for a precise definition).
\begin{thm}
All subsequential limits of $\{\gamma_n\}$ under $d_{Re}$ coincide almost surely.
\end{thm}
\begin{proof}
    Suppose, for contradiction, that there exist two subsequences $\gamma_{k_n}$ and $\gamma_{l_n}$ converging under $d_{Re}$ to distinct limits $\widehat \gamma_1$ and $\widehat \gamma_2$ respectively (for a positive probability). By Corollary 5.3 (i) of \cite{VCL16}, the random walk frontiers converge in distribution to the Brownian frontier under the Hausdorff metric $d_H$. Since $d_{Re}$ refines $d_H$, the traces of $\widehat \gamma_1$ and $\widehat \gamma_2$ must coincide as closed sets in $\Rb^2$. Consequently, either $\widehat \gamma_1$ or $\widehat \gamma_2$ contains a backtracking. Here, by backtracking of $\widehat\gamma_i$, we mean the existence of three distinct time parameters $s_1<s_2<s_3$ and a decreasing continuous bijection $\theta:[s_1,s_2]\to[s_2,s_3]$ such that $\widehat\gamma_i(t) = \widehat\gamma_i(\theta(t))$.

    \noindent For $t_0\ge0$ and $A\subseteq \Rb^2$, we introduce the following notation to denote the first hitting times:
    $$
    \tau_A(t_0) =\inf \{ t>t_0:S(t) \in A \},\ \ \inf \emptyset = \infty.
    $$
    For $C_1,C_2>0$, construct the following stopping times (see Fig. \ref{fig:1350}):
    \begin{itemize}
	\item Let $\tau_1 = \tau_{\partial B(x,C_1)}(0)$, $\tau_2 = \tau_{\partial B(x,e^{-4n/5})}(\tau_1)$, and $\sigma_1 = \tau_{\partial B(x,C_2 e^{-n})}(\tau_2)$.
	\item Let $\tau_3 = \tau_{\partial B(x, e^{-4n/5})}(\sigma_1)$, $\tau_4 = \tau_{\partial B(x, C_1 )}(\tau_3)$, and $\sigma _2 = \tau_{\partial B(x,C_1)}(\tau_4)$.
	\item Let $\tau_5 = \tau_{\partial B(x, C_1 )}(\sigma_2)$, $\tau_6 = \tau_{\partial B(x, e^{-4n/5})}(\tau_5)$, and $\sigma_3 = \tau_{\partial B(x, C_2 e^{-n})}(\tau_6)$,
	\item Let $\tau_7 = \tau_{\partial B(x, e^{-4n/5})}(\sigma_3)$ and $\tau_8 = \tau_{\partial B(x, C_1 )}(\tau_7)$.
    \end{itemize}
    This generates four crossing arms $l_1=S[\tau_1,\tau_2],\ l_2=S[\tau_3,\tau_4],\ l_3=S[\tau_5,\tau_6],\ l_4=S[\tau_7,\tau_8]$ within the annulus $\annulus(x,e^{-4n/5},C_1 )$.
    \begin{figure}[H]
    \centering
    \includegraphics[width=0.6\linewidth]{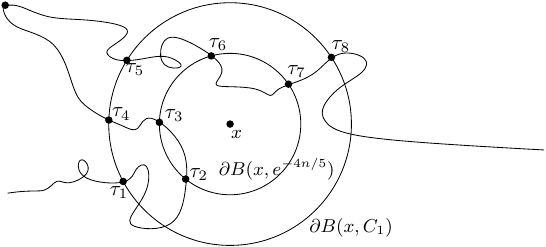}
    \caption{Illustration of backtracking and four arms crossing the annulus. Note that typically these arms intersect each other (but we choose not to draw this scenario to avoid a messy figure).}
    \label{fig:1350}
    \end{figure}
    A disk $B(x,C_1)\subseteq B(0,1)$ is \textit{bad} if the above stopping times are well defined, less than $\infty$ and there exists a permutation $\pi$ such that
	\begin{equation}\label{eq1130}
	    (l_{\pi(1)} \cup l_{\pi(4)}) \cap (l_{\pi(2)} \cup l_{\pi(3)}) = \emptyset\mathrm{\ for\ some\ permutation\ }\pi\ \mathrm{on}\ \{1,2,3,4\}.
	\end{equation}
    According to Lemma 3.6 of \cite{KM10} and using the intersection exponent $\xi(2,2) = 35/12$ from \eqref{eq:intersection exponents}, we obtain the following estimate of the probability that $B(x,C_1)$ is \textit{bad}.
    \begin{equation}
        \Pb\big(B(x,C_1)\ \mathrm{is\ }\textit{bad}\big)\asymp (e^{-4n/5}/C_1)^{\xi(2,2)}\asymp e^{-7n/3}.
    \end{equation}
    The expected number of bad disks in $B(0,1)$ is $O( e^{2n}\times e^{-7n/3}) = O(e^{-n/3})\to 0$ as $n\to\infty$. However, the existence of backtracking of $\wh \gamma_i$ implies that the frontier must enter and exit such a disk twice, creating four arms (of the frontier) satisfying \eqref{eq1130}. 
    Then, one can recognize four arms of the original simple random walk satisfying \eqref{eq1130}, which guarantees the existence of an arbitrary small bad disk. 
    Thus, the expected number of bad disks should have a positive lower bound. Hence, under the KMT coupling \eqref{eq:se1}, backtracking paths are excluded almost surely. 
\end{proof}

\subsection{Proof of Theorem \ref{thm:convergence under np}}\label{sec6.2:Proof of Theorem 1.1}

\subsubsection{Tightness}
In this subsection, we prove the tightness of $\{\gamma_n\}$ with respect to the natural parametrization metric $\rho$. We will make use of the established convergence under reparametrization and convergence of occupation measure.
\begin{thm}
	The sequence $\{\gamma_n\}$ is tight with respect to $\rho$.
\end{thm}
\begin{proof}
    We prove by contradiction. Suppose $\{\gamma_n\}$ is not tight, then there exist $\delta_0>0$, a subsequence $k(n) \to \infty$ and $\varepsilon_n\downarrow 0$ such that
\[
\Pb\big(\mathrm{osc}(\varepsilon_n,\gamma_{k(n)})>\delta_0\big)>\delta_0,
\]
where $\mathrm{osc}(\varepsilon,f) = \sup_{|t-s|<\varepsilon}|f(t)-f(s)|$ denotes the oscillation of $f$ in a small time $\varepsilon$. 
For each $n$, there exist $s_n$ and $t_n$ such that $|s_n-t_n|<\varepsilon_n$ and $|\gamma_{k(n)}(s_n)-\gamma_{k(n)}(t_n)|>\delta_0$. We can pick a subsequence of $(s_n,t_n,\gamma_{k(n)}(s_n),\gamma_{k(n)}(t_n))$ (still indexed by $n$) such that 
\[
s_n\ \mathrm{and}\ t_n\  \mathrm{have\ the\ same\ limit},\ 
\gamma_{k(n)}(s_n)\to x,\mathrm{\ and\ }\gamma_{k(n)}(t_n)\to y,
\]
where $x$ and $y$ are two different points in $\overline{\Dc}$. By Theorem \ref{thm:main result 3} (reparametrization convergence), $x$ and $y$ lie on $\widetilde \gamma$. Otherwise, suppose $x\notin\widetilde\gamma$, then there exists an open set $O$ containing $x$ such that $\widetilde \gamma \cap O=\emptyset$, which leads to a contradiction because this implies $\gamma_n \cap O=\emptyset$ (for $n$ large enough) by convergence under reparametrization, contradicting to the definition of $x$. Let $x=\widetilde \gamma(s)$ and $y=\widetilde \gamma(t)$ with $s<t$.

Select a rational disk $B=B(x_0,r_0)$ such that 
\begin{align}\label{eq:1459}
	\gamma_{k(n)}\cap B(x_0,r_0/2) \neq\emptyset,\ \widetilde \gamma \cap B(x_0,r_0/2) \neq\emptyset,
	\ \widetilde\gamma[0,s]\cap B = \emptyset,\ \mathrm{and}\ \widetilde\gamma[t,t_{\widetilde \gamma}]\cap B= \emptyset,
\end{align}
for all large $n$. 
Since $\widetilde\gamma$ enters $B(x_0,r_0/2)$, its occupation measure satisfies $\nu(B(x_0,2r_0/3))>0$. However, by the KMT coupling \eqref{eq:se1} and conditions \eqref{eq:1459}, we know that $$\gamma_{k(n)}[0,s]\cap B(x_0,2r_0/3)=\emptyset\ \ \mathrm{and}\ \ \gamma_{k(n)}[t,t_{\gamma_{k(n)}}]\cap B(x_0,2r_0/3)=\emptyset$$ for $n$ large enough, implying $\nu_n(B(x_0,2r_0/3))\to 0$. This contradicts \eqref{eq:convergence of indicator}, as $\nu_n(V)\to \nu(V)$ in probability for any rational disk $V$, and every subsequence $\nu_{k(n)}(V)$ contains a further almost surely convergent subsequence. Thus, the original assumption of non-tightness is false. 
\end{proof}

\subsubsection{Uniqueness}
By tightness, each subsequence of $\{\gamma_n\}$ has a further convergent subsequence under the natural parametrization metric $\rho$. It remains to show the uniqueness of the subsequential limit. We begin by excluding a situation where the subsequential limit $\lim_{n \to \infty}\gamma_{k(n)}$ stops at a point for a positive time.
\begin{lemma}
    The subsequential limit $\lim_{n \to \infty}\gamma_{k(n)}$ cannot remain at any single point over a positive time interval. 
\end{lemma}
\begin{proof}
    By Theorem \ref{thm:418}, the continuum occupation measure $\nu$ is non-atomic, implying that $\wt\gamma$ cannot stay at any point for a positive time. By convergence of occupation measure, the subsequential limit $\lim_{n \to \infty}\gamma_{k(n)}$ cannot remain at any single point over a positive time interval. 
\end{proof}
Finally, we establish the uniqueness of subsequential limits in the natural parametrization metric $\rho$, thereby proving Theorem \ref{thm:convergence under np}. Our idea is to prove that the subsequential limit must be the Brownian frontier denoted by $\widetilde \gamma$. 
\begin{thm}
	If the subsequential limit $\lim_{n \to \infty}\gamma_{k(n)}$ exists with respect to $\rho$, then it is equal to $\widetilde{\gamma}$.
\end{thm}
\begin{proof}
    Denote by $E$ the event that there exists a subsequence $\{\gamma_{k_n}\}$ converging to a limit $\widehat \gamma \neq\widetilde \gamma$ under the natural parametrization metric $\rho$. Assume for contradiction that $\Pb(E)>0$. By Theorem \ref{thm:main result 3}, $\widehat \gamma$ and $\widetilde \gamma$ share identical trajectories but differ in temporal parametrization.
    
\noindent \textbf{Case 1.} There exists a rational disk $B = B(x_0,r_0)$ such that 
\[\widetilde \gamma \cap B(x_0,r_0/2) \neq\emptyset \mathrm{\ \ and\ \ }
\int\ind{\widetilde \gamma(s)\in B}ds>\int\ind{\widehat \gamma(s)\in B}ds.
\]
By \eqref{eq1:nu(V)} and \eqref{eq:frontier green's function}, there is no occupation measure on $\partial B$. We can pick $\eps>0$ such that 

\begin{align*}
	\int\ind{\widetilde \gamma(s)\in B_{1-\eps}}ds>(1+\eps)\int\ind{\widehat \gamma(s)\in B}ds
	\ge (1+\eps)\int\ind{ \gamma_{k(n)}(s)\in B_{1-\eps}}ds,
\end{align*}
where $B_{1-\eps} = B(x_0,(1-\eps)r_0)$, implying $\nu_{k(n)}(B_{1-\eps})<\widetilde \nu(B_{1-\eps})/(1+\eps)$.

\noindent \textbf{Case 2.} There exists a rational disk $B$ such that 
\[
\int\ind{\widetilde \gamma(s)\in B}ds<\int\ind{\widehat \gamma(s)\in B}ds.
\]
Again, since there is no occupation measure on $\partial B$, we can pick $\eps>0$ such that

\begin{align*}
(1+\eps)\int\ind{\widetilde \gamma(s)\in B_{1+\eps}}ds&<\int\ind{\widehat \gamma(s)\in B}ds<\int\ind{\gamma_{k(n)}\in B_{1+\eps}}ds,
\end{align*}
where $B_{1+\eps} = B(x_0,(1+\eps)r_0)$, implying $\nu_{k(n)}(B_{1+\eps})>(1+\eps)\widetilde \nu(B_{1+\eps})$.

On the event $E$, either case produces a disk $V$ (dependent on the sample), where any subsequence of $\nu_{k(n)}(V)$ fails to converge to $\widetilde \nu(V)$. This violates \eqref{eq:convergence of indicator}, which guarantees almost sure convergence of a subsequence of $\nu_{k(n)}(V)$ to $\widetilde \nu(V)$. 
\end{proof}

\appendix
\section{Proof of Proposition \ref{prop:Dnf}}\label{appen:A}
\begin{proof}[Proof of Proposition \ref{prop:Dnf}]
We decompose the target event $ D_{n}(\ol{\initialconfig})\,\cap\, \{ \ol\initialconfig\oplus\ol {S}[0,\tau_{m/2}]\notin \mathrm{NICE}_{m/2} \} $ into five sub-events:
\begin{itemize}
    \item $E_m^1:=\{0\notin \fr(\mathrm{THICK}_m(\initialconfig_1\oplus S_1[0,\tau_n]))\}$ happens for some $29n/60\le m\le n/2$.
    \item $E_m^2:=\{0\notin \fr(\mathrm{THICK}_m(\initialconfig_2\oplus S_2[0,\tau_n]))\}$ happens for some $29n/60\le m\le n/2$.
    \item $E_m^3:=\{0\notin\fr(\mathrm{THICK}_m(\initialconfig_1\oplus S_1[0,\tau_n])\cup(\initialconfig_2\oplus S_2[0,\tau_n]))\}$ happens for some $29n/60\le m\le n/2$.
    \item $E_m^4:=\{0\notin\fr((\initialconfig_1\oplus S_1[0,\tau_n])\cup \mathrm{THICK}_m(\initialconfig_2\oplus S_2[0,\tau_n]))\}$ happens for some $29n/60\le m\le n/2$.
    \item $E_{m,m'}^5:=\{0\notin\fr(\mathrm{THICK}_m(\initialconfig_1\oplus S_1[0,\tau_n])\cup\mathrm{THICK_{m'}}(\initialconfig_2\oplus S_2[0,\tau_n]))\}$ happens for some $29n/60\le m,m'\le n/2$.
\end{itemize}
The following three auxiliary lemmas demonstrate that the probabilities of the events $D_n(\ol\initialconfig)\cap E_m^i$ and $D_n(\ol\initialconfig)\cap E_{m,m'}^5$ are negligible relative to the probability of $D_n(\ol\initialconfig)$. By summing over all $m,m'$ and combining the contributions from the five cases, we thereby complete the proof.
\end{proof}
\begin{lemma}\label{lem:A.1}
    There exist constants $c,u>0$, such that for all $n-2\ge m-1\ge 10l$ and $\ol\initialconfig \in\Yc_l$, 
    \begin{equation}
        \Pb\big(D_n(\ol\initialconfig)\cap E_m^1\big)\le ce^{-um}e^{-\alpha(n-l)}.
    \end{equation}
\end{lemma}
\begin{proof}
Since $\initialconfig_1 \oplus S_1[0,\tau_n]$ does not disconnect the origin, while the thickened curve does, there exists a small exit of diameter $ce^{15m/16}$, where the origin is connected to $\infty$, and this exit would close if the SRW path is thickened. Thus, we consider the stopping time $\tau^*>\tau_{m-1}$ such that $S_1[0,\tau^*] \cup D(S_1(\tau^*),ce^{15m/16})$ disconnects the origin from $\infty$. The exponential penalty comes from the fact that the single arm $S_1[\tau^*,\tau_n]$ should not disconnect this circle $\partial D(S_1(\tau^*),ce^{15m/16})$ (otherwise, $D_n(\ol\initialconfig)$ will not happen). More precisely, we do the following decomposition.
\begin{itemize}
    \item Let $(\eta_1,\eta_2)$ be a pair of non-disconnecting paths starting from the endpoint of $\initialconfig_1$ and $\initialconfig_2$ and stopped when reaching the circle $\partial\Dc_{m-1}$. By \eqref{eq:rw-disc-exp}, the total mass is $\asymp e^{-\alpha(m-1-l)}$.
    \item Let $\eta_3^1$ be a path starting from the endpoint of $\eta_1$, and stopped when there exists a stopping time $\tau^*$ such that $\initialconfig_1\cup\eta_1\cup\eta_3^1\cup D(S_1(\tau^*),ce^{15m/16})$ disconnects the origin. 
    The total mass is $\lesssim 1$. 
    \item Let $\eta_3^2$ be a one-arm non-disconnecting path starting from the endpoint of $\eta_3^1$ and stopped when reaching the circle $\partial D(z,ce^{31m/32})$. By \eqref{eq:one-arm event}, the total mass of $\eta_3^2$  not disconnecting $\partial D(z,ce^{15m/16})$ from $\infty$ is $\lesssim e^{-m/128}$.
    \item Let $\eta_3^3$ be a path starting from the endpoint of $\eta_3^2$ and stopped when reaching the circle $\partial \Dc_{m+1}$. The total mass is $\lesssim 1$.
    \item Let $\eta_4$ be a path starting from the endpoint of $\eta_2$ and stopped when reaching the circle $\partial \Dc_{m+1}$. The total mass is $\lesssim 1$.
    \item Let $(\eta_5,\eta_6)$ be a pair of non-disconnecting paths starting from the endpoint of $\eta_3^3$ and $\eta_4$ and stopped when reaching the circle $\partial\Dc_{n}$. By \eqref{eq:rw-disc-exp}, the total mass of $\eta_5\cup\eta_6$  not disconnecting the origin from $\infty$ is $\asymp e^{-\alpha (n-m-1)}$.
\end{itemize}
Multiplying the above factors, the total mass of case 1 is $\asymp e^{-\alpha(m-1-l)}\times e^{-m/128}\times e^{-\alpha(n-m-1)}$. 
We conclude the proof.
\end{proof}
\begin{figure}[H]
    \centering
    \includegraphics[width=0.5\linewidth]{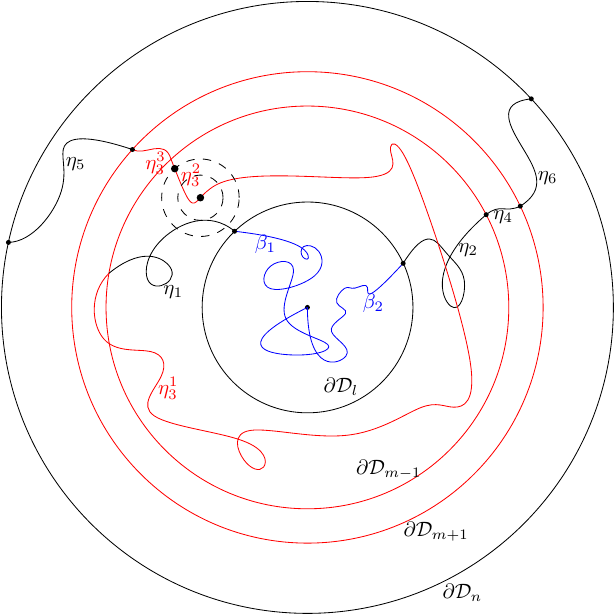}
    \caption{Lemma A.1.\ The radii of two dashed circles are $ce^{15m/16}$ and $ce^{31m/32}$ respectively.}
    \label{fig:3.11 case1.jpg}
\end{figure}

\begin{lemma}
    There exist constants $c,u>0$, such that for all $n-2\ge m-1\ge 10l$ and $\ol\initialconfig \in\Yc_l$, 
    \begin{equation}
    \Pb\big(
D_n(\ol\initialconfig)\cap (E_m^1)^c \cap E_m^3\big)
\le c e^{-um}e^{-\alpha(n-l)}.
    \end{equation}
\end{lemma}
\begin{proof}
Define $E_1=\{S_1[0,\tau_{m-1}]\cap S_2[0,\tau_{m-1}] = \emptyset\}$, then 
$$\Pb\big(D_n(\ol\initialconfig) \cap E_1\big)\le ce^{-\xi(1,1)(m-1-l)}e^{-\alpha(n-m+1)}\le ce^{-um}e^{-\alpha(n-l)},$$
where we use the independence of $S_i[0,\tau_{m-1}]$ and $S_i[\tau_{m-1},\tau_{n}]$ by the strong Markov property, and $\xi(1,1)=5/4$ is the intersection exponent. Thus, it suffices to consider the non-disconnecting paths with only one opening. Analogous to Lemma \ref{lem:A.1}, the exponential penalty also comes from the small exit generated by $S_1$ and $S_2$. More precisely, we do the following path decomposition.
\begin{itemize}
    \item Let $(\eta_1,\eta_2)$ be a pair of non-disconnecting but intersected paths starting from the endpoint of $\initialconfig_1$ and $\initialconfig_2$ and stopped when reaching the circle $\partial \Dc_{m-1}$. By \eqref{eq:rw-disc-exp}, the total mass is $\asymp e^{-\alpha(m-1-l)}$.
    \item Let $(\eta_5,\eta_6)$ be a pair of non-disconnecting paths starting from the circle $\partial \Dc_{m+1}$ and stopped when reaching the circle $\partial \Dc_n$. By \eqref{eq:rw-disc-exp}, the total mass of $\eta_5\cup\eta_6$ not disconnecting the origin from infinity is $\lesssim e^{-\alpha({n-m-1})}$.
    \item Let $\eta_4$ be a bridge starting from the endpoint of $\eta_2$ and stopped when reaching the starting point of $\eta_6$. The total mass is $\lesssim 1$.
    \item Let $\eta_3^1$ be a path starting from the endpoint of $\eta_1$ and stopped when there exists a disk $D(z,ce^{15m/16})$ such that $(\initialconfig_1\cup\eta_1\cup\eta_3^1)\cup(\initialconfig_2\cup\eta_2\cup\eta_4)\cup D(z,ce^{15m/16})$ disconnects the origin from infinity. The total mass is $\lesssim 1$.
    \item Let $\eta_3^2$ be a one-arm non-disconnection path starting from the endpoint of $\eta_3^1$ and stopped when reaching the circle $\partial D(z,ce^{31m/32})$. By \eqref{eq:one-arm event}, the total mass of $\eta_3^2$  not disconnecting $D(z,ce^{15m/16})$ from infinity is $\lesssim e^{-m/128}$.
    \item Let $\eta_3^3$ be a bridge starting from the endpoint of $\eta_3^2$ and stopped when reaching the starting point of $\eta_5$. The total mass is $\lesssim 1$. 
\end{itemize}
Multiplying the above factors, the total mass is $\asymp e^{-\alpha(m-1-l)}\times e^{-m/128}\times e^{-\alpha(n-m-1)}$. We obtain an extra cost for $\Pb(D_n(\ol\initialconfig)\cap (E_m^1)^c \cap E_m^3)$ and conclude the proof.
\end{proof}
\begin{figure}[H]
    \centering
    \includegraphics[width=0.5\linewidth]{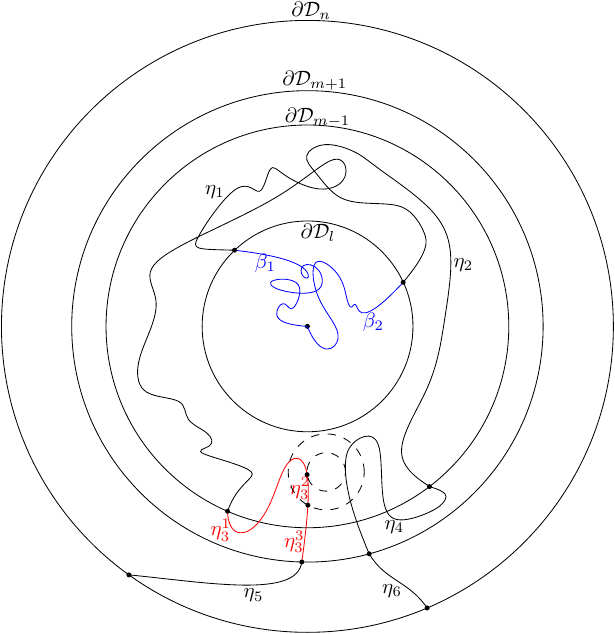}
    \caption{Lemma A.2.\ The radii of two dashed circles are $ce^{15m/16}$ and $ce^{31m/32}$ respectively.}
    \label{fig:3.12.jpg}
\end{figure}
\begin{lemma}There exist constants $c,u>0$, such that for all $n-2\ge m'-1\ge m-1\ge 10l$ with $m'\asymp m$ and $\ol\initialconfig \in\Yc_l$,
    \begin{equation}
        \Pb\big(
D_n(\ol\initialconfig)\cap (E_m^1)^c \cap (E_m^2)^c \cap (E_m^3)^c \cap (E_m^4)^c \cap E_{m,m'}^5
\big)
\le c e^{-um}e^{-\alpha(n-l)}.
    \end{equation}
\end{lemma}
\begin{proof}
    Define $E_1=\{S_1[0,\tau_{m-1}]\cap S_2[0,\tau_{m-1}] = \emptyset\}$, then 
$$\Pb\big(D_n(\ol\initialconfig) \cap E_1\big)\le ce^{-\xi(1,1)(m-1-l)}e^{-\alpha(n-m+1)}\le ce^{-um}e^{-\alpha(n-l)},$$
where we use the independence of $S_i[0,\tau_{m-1}]$ and $S_i[\tau_{m-1},\tau_{n}]$ by the strong Markov property, and $\xi(1,1)=5/4$ is the intersection exponent. Thus, it suffices to consider the non-disconnecting paths with only one opening. 
\begin{itemize}
    \item Let $(\eta_1,\eta_2^1)$ be a pair of non-disconnecting but intersected paths starting from the endpoint of $\initialconfig_1$ and $\initialconfig_2$ and stopped when reaching the circle $\partial \Dc_{m-1}$. By \eqref{eq:rw-disc-exp}, the total mass is $\asymp e^{-\alpha(m-1-l)}$.
    \item Let $\eta_2^2$ be a path starting from the endpoint of $\eta_2^1$ and stopped when reaching the circle $\Dc_{m'-1}$. The total mass is $\lesssim 1$. 
    \item Let $(\eta_3^1,\eta_4^1)$ be two paths starting from the endpoint of $(\eta_1,\eta_2^2)$ and stopped when there exists a disk $D(z,e^{15m/16}+e^{15m'/16})$ such that $(\eta_1\cup\eta_3)\cup(\eta_2^1\cup\eta_2^2\cup\eta_4)\cup D(z,e^{15m/16}+e^{15m'/16})$ disconnects the origin from infinity. The total mass is $\lesssim 1$.
    \item Let $(\eta_3^2,\eta_4^2)$ be a pair of non-disconnecting paths starting from the endpoint of $(\eta_3^1,\eta_4^1)$ and stopped when they reach the circle $\partial D(z,e^{m/32}(e^{15m/16}+e^{15m'/16}))$. The total mass of $\eta_3^2\cup\eta_4^2$ not disconnecting the circle $D(z,e^{15m/16}+e^{15m'/16})$ and infinity is $\lesssim e^{-m/128}$.
    \item Let $(\eta_3^3,\eta_4^3)$ be two paths starting from the endpoint of $(\eta_3^2,\eta_4^2)$ and stopped when reaching the circle $\partial \Dc_{m+1}$. The total mass is $\lesssim 1$.
    \item Let $(\eta_5,\eta_6)$ be a pair of non-disconnecting paths starting from the circle $\partial \Dc_{m+1}$ and stopped when reaching the circle $\partial \Dc_n$. By \eqref{eq:rw-disc-exp}, the total mass that $\eta_5\cup\eta_6$ does not disconnect the origin from infinity is $\lesssim e^{-\alpha({n-m-1})}$.
\end{itemize}
Multiplying the above factors, the total mass is $\asymp e^{-\alpha(m-1-l)}\times e^{-m/128}\times e^{-\alpha(n-m-1)}$. We obtain an extra cost for $\Pb(D_n(\ol\initialconfig)\cap (E_m^1)^c \cap (E_m^2)^c \cap (E_m^3)^c \cap (E_m^4)^c \cap E_{m,m'}^5)$ and conclude the proof.
\end{proof}
\begin{figure}[H]
    \centering
    \includegraphics[width=0.5\linewidth]{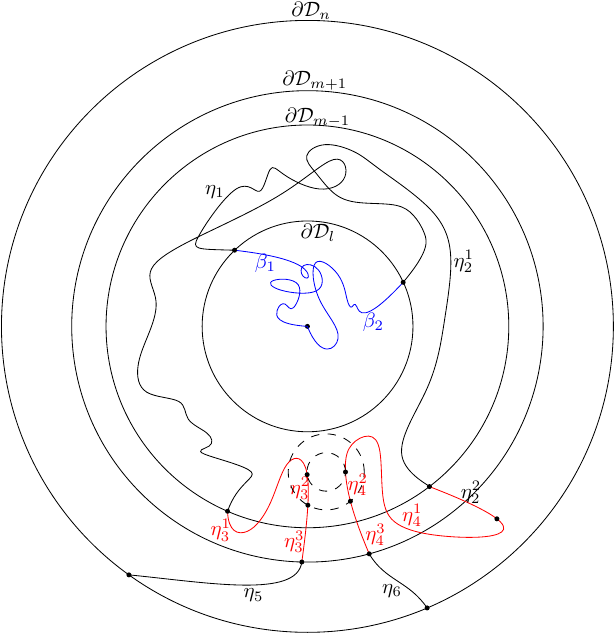}
    \caption{Lemma A.3.\ The radii of two dashed circles are $e^{15m/16}+e^{15m'/16}$ and $e^{m/32}(e^{15m/16}+e^{15m'/16})$ respectively.}
    \label{fig:3.13.jpg}
\end{figure}
\section{Proof of Proposition \ref{prop:Dnf2}}\label{appen:B}
\begin{proof}[Proof of Proposition \ref{prop:Dnf2}]
	Define $E_1=D_{n}(\ol\zeta)  \setminus  \widetilde D_n(\mathrm{THICK}_{\lfloor 29m/60\rfloor,m/2}(\ol\zeta))$. Using the KMT coupling in \eqref{eq:se11}, we couple $S_1$ with $W_1$ and $S_2$ with $W_2$, maintaining proximity within $Kn$ until exiting $\Dc_n$. This yields 
	\[
	 \widetilde D_n(\mathrm{THICK}_{\lfloor 29m/60\rfloor,m/2}(\ol\zeta))^c\subseteq \bigcup_{s=m/2}^{n}\bigcup_{s'=m/2}^{n}F_{s,s'},
	\]
    where $$F_{s,s'} = \big\{0\notin\fr\big(\mathrm{THICK}_s(\zeta_1\oplus S_1[0,\tau_n])\cup\mathrm{THICK}_{s'}(\zeta_2\oplus S_2[0,\tau_n]) \big)\big\},$$
    noting that $2Kn<e^{15m/32}\le\min\{e^{15s/16},e^{15s'/16}\}$. Following the methodology in the proof of Proposition \ref{prop:Dnf}, we have for $m/2\le s,s'\le n$,
	\[
	\Pb(D_{n}(\ol{\initialconfig})\,\cap\, F_{s,s'} )\le ce^{-um}e^{-\alpha (n-m/2)}.
	\]
	Summing over $s,s'$ gives $$\Pb(E_1)\le ce^{-um}e^{-\alpha(n-m/2)}.$$ Similarly, defining $E_2 =  \widetilde D_n(\mathrm{THICK}_{\lfloor 29m/60\rfloor,m/2}(\ol\zeta))\setminus D_{n}(\ol\zeta) $, we obtain $$\Pb(E_2)\le ce^{-um}e^{-\alpha(n-m/2)}.$$ This completes the proof of \eqref{eq:779}.
\end{proof}
\section{Proof of Proposition \ref{prop:536}}\label{appen:C}
\paragraph{Quasi-invariant measure from the origin.} For any two compact sets $K_1,K_2$, two points $x_1\in K_1,x_2\in K_2$ and $r>0$ such that $K_1\cup K_2\subseteq\Dc_r$, define the non-disconnecting event:
\[
A_r(\ol K,\ol x) = \Big\{0\in\fr\big((W_1[0,T_{\partial\Dc_r}]\cup K_1)\cup(W_2[0,T_{\partial\Dc_r}]\cup K_2)\big)\Big\},
\]where $W_1,W_2$ are two independent Brownian motions starting from $x_1,x_2$ respectively. 
Let $Q_r(\ol K,\ol x)$ denote the distribution of $e^{-r}(W_1,W_2)$ conditionally on $A_r(\ol K,\ol x)$. Define the non-disconnecting paths from the origin as:
\[
\wt \Xc := \big\{(\beta_1,\beta_2)\in \wt\Gamma_{0,\partial \Dc}^\Dc \times\wt\Gamma_{0,\partial \Dc}^\Dc :0\in\fr(\beta_1\cup\beta_2)\big\}.
\]
The quasi-invariant measure (from the origin) $\mathbf Q$ is a probability measure on $\wt\Xc$ (see e.g.\ \cite{L98} for existence) such that for some $u>0$:
\begin{equation}\label{eq:1229}
    \left \|Q_r(\ol K,\ol x)[T_{-r/2},T_0]-\mathbf Q[T_{-r/2},T_0]\right\|_\mathrm{{TV}} = O(e^{-ur}),
\end{equation}
uniformly in $(\ol K, \ol x)$. Here, $\left\|\cdot\right\|_\mathrm{TV}$ is the total variance norm, and $Q[T_{-r/2},T_0]$ denotes the distribution of $\ol\beta[T_{-r/2},T_0]$ for $\ol\beta\sim Q$.

\textbf{Quasi-invariant measure from infinity.} For any two compact sets $K_1,K_2$, two points $x_1\in K_1,x_2\in K_2$ and $r>0$ such that $K_1\cup K_2\subseteq\Dc_r^c$, define the non-disconnecting event:
\[
A_r^*(\ol K,\ol x) = \Big\{0\in\fr\big((W_1[0,T_{0}]\cup K_1)\cup(W_2[0,T_{0}]\cup K_2)\big)\Big\}\bigcap\Big\{T_0(W_1),T_0(W_2)<\infty\Big\},
\]where $W_1,W_2$ are two independent Brownian motions starting from $x_1,x_2$ respectively. Let $Q_r^*(\ol K,\ol x)$ be the probability measure of $(W_1[T_{r/2},T_0],W_2[T_{r/2},T_0])$, conditioned on the event $A_r^*(\ol K,\ol x)$. Let $\wt \Gamma_{\infty,\partial \Dc}^{\Dc^c}$ denote the collection of all paths starting from $\infty$ and stopped when reaching the unit circle. Define the non-disconnecting paths from infinity as follows:
\[
\wt\Xc^*:=\big\{(\beta_1,\beta_2)\in \wt \Gamma_{\infty,\partial \Dc}^{\Dc^c}\times\wt \Gamma_{\infty,\partial \Dc}^{\Dc^c}:0\in\fr(\beta_1\cup\beta_2)\big\}.
\] 
The quasi-invariant measure (from infinity) $\mathbf Q^*$ is a probability measure on $\wt\Xc^*$ such that for some $u>0$,
\begin{equation}\label{eq:C.2}
    \left \|Q_r^*(\ol K,\ol x)[T_{r/2},T_0]-\mathbf Q^*[T_{r/2},T_0]\right\|_\mathrm{{TV}} = O(e^{-ur}),
\end{equation}uniformly in $(\ol K,\ol x)$. Here, $Q[T_{r/2},T_0]$ denotes the distribution of $\ol\beta[T_{r/2},T_0]$ for $\ol\beta\sim Q$.
\begin{proof}[Proof of Proposition \ref{prop:536}]
    Let $(\beta_1,\beta_2)\in \wt\Gamma^{\Dc_n}_{x_1,0}\times\wt\Gamma^{\Dc_n}_{x_2,\partial\Dc_n}$ such that $z_n\in\fr((\wt\zeta_1\cup\initialconfig_1)\cup(\wt\zeta_2\cup\initialconfig_2))$. We do the following path decomposition:
    $$\beta_1 = \eta_1\oplus\omega\oplus \eta_2,\ \beta_2=\eta_1'\oplus\omega\oplus \eta_2',$$where $\eta_1$ is the segment of $\beta_1$ from $0$ to its first visit of $\partial \Dc_{n/2+1}(z_n)$, $\eta_2$ is the segment of $\beta_1^R$ from $x_1$ to its first visit of $\partial \Dc_{n/2}(z_n)$, and $\omega$ is the intermediate part of $\beta_1$. Decompose $\beta_2$ in the same way (with 0 replaced by $\partial \Dc_n$).

    We first sample $\ol\eta_1 = (\eta_1,\eta_1')$ according to the probability measure $\wt\mu_1/\left\|\wt\mu_1\right\|$, where
    \[
    \widetilde\mu_1=\int_{\partial \mathcal{D}_{n/2+1}(z_n)}\int_{\partial \mathcal{D}_{n/2+1}(z_n)}\mu_{0,y_1}^{\Dc_n\setminus\Dc_{n/2+1}(z_n)}\otimes\mu_{\partial \Dc_n,y_1'}^{\Dc_n\setminus\Dc_{n/2+1}(z_n)}[z_n\in\fr(\eta_1\cup\eta_1')]\sigma(dy_1,dy_1').
    \]
    By \eqref{eq:1229}, the distribution of $\ol\eta_1[T_{3n/4+1},T_{n/2+1}]$ has total variation distance $O(e^{-ur})$ to the probability measure $\mathbf Q_{n/2+1}^*(z_n)[T_{3n/4+1},T_{n/2+1}]$, where $\mathbf Q_{n/2+1}^*(z_n)$ is the push-forward of the quasi-invariant measure of $\mathbf Q^*$ under the map $x\mapsto e^{n/2}x+z_n$. To derive the total mass of $(\eta_1,\eta_1')$, we decompose $\eta_1$ and $\eta_1'$ according to their first visit of $D(z_n,ce^n)$. An estimate analogous to (6.9) in \cite{GLPS23}, combined with \eqref{eq:259}, implies that the total mass of $(\eta_1,\eta_1')$ is $\simeq ce^{-\alpha n/2}G_{\mathcal{D}}^{\fr}(z)$.

    Next, we sample $(\eta_2,\eta_2')$ according to the following law:
    \[
    e^{n/2}\circ Q_{n/3}(e^{-n/6}\wt\zeta,e^{-n/6}\ol x),
    \]where $\widetilde\zeta = (\widetilde\zeta_1,\widetilde\zeta_2)$. By \eqref{eq:C.2}, the law of $\ol\eta_2[T_{n/3},T_{n/2}]$ has total variation distance $O(e^{-ur})$ to the probability measure $\mathbf Q_{n/2}(z_n)$, where $\mathbf Q_{n/2}(z_n)$ is the push-forward of quasi-invariant measure $\mathbf Q$ under the map $x\mapsto e^{n/2}x+z_n$. The total mass of $(\eta_2,\eta_2')$ is $\Pb(\widetilde D_{n/2}(\mathrm{THICK}_{\lfloor 29n/180\rfloor,n/6}(\ol\zeta)))$.

    Finally, we consider the intermediate parts $(\omega,\omega')$. For any given $(\ol\eta_1,\ol\eta_2)$, we define the set
    \[
    \Wc(\ol\eta_1,\ol\eta_2):=\Big\{(\omega,\omega')\in \wt\Gamma_{y_1,y_2}^{\Dc_n}\times\wt\Gamma_{y_1',y_2'}^{\Dc_n}:z_n\in\fr\big((\eta_1\oplus\omega\oplus\eta_2)\cup(\eta_1'\oplus\omega'\oplus\eta_2')\big)\Big\}.
    \]
    Furthermore, we introduce the subset of $\Wc(\ol\eta_1,\ol\eta_2)$:
    \[
    \Wc'(\ol\eta_1,\ol\eta_2):=\Big\{(\omega,\omega')\in\Wc(\ol\eta_1,\ol\eta_2):\omega,\omega'\in\Dc_{3n/4+1}(z_n)\setminus\Dc_{n/3}(z_n)\Big\}.
    \]
    By \eqref{one-arm disconnection exponent}, we obtain the asymptotic equivalence
    \[
    \mu_{y_1,y_2}^{\Dc_n}\otimes\mu_{y_1',y_2'}^{\Dc_n}\big(\Wc(\ol\eta_1,\ol\eta_2)\big)\simeq\mu_{y_1,y_2}^{\Dc_n}\otimes\mu_{y_1',y_2'}^{\Dc_n}\big(\Wc'(\ol\eta_1,\ol\eta_2)\big).
    \]
    Let $\ol E$ denote the probability distribution of $(\ol\eta_1,\ol\eta_2)$ within the annulus $\Dc_{3n/4+1}(z_n)\setminus\Dc_{n/3}(z_n)$. We have established that $\ol E$ has total variance distance $O(e^{-ur})$ to the product measure $\mathbf Q_{n/2+1}^*(z_n)\otimes\mathbf Q_{n/2}(z_n)$. Therefore, 
    \begin{equation}\label{eq:1351}
    \ol E[\mu_{y_1,y_2}^{\Dc_n}\otimes\mu_{y_1',y_2'}^{\Dc_n}\big(\Wc(\ol\eta_1,\ol\eta_2)\big)]\simeq \mathbf Q_{n/2+1}^*(z_n)\otimes\mathbf Q_{n/2}(z_n)[\mu_{y_1,y_2}^{\Dc_n}\otimes\mu_{y_1',y_2'}^{\Dc_n}\big(\Wc'(\ol\eta_1,\ol\eta_2)\big)].
    \end{equation}
    By scaling and translation invariance of the quasi-invariant measures, the right hand side of \eqref{eq:1351} converges to a constant number $\wh c$, which is bounded away from 0 by separation lemmas and bounded away from infinity by \eqref{one-arm disconnection exponent} and standard estimates of Green's functions. By multiplying the masses $ce^{-\alpha n/2}G_{\mathcal{D}}^{\fr}(z)$, $\Pb(\widetilde D_{n/2}(\mathrm{THICK}_{\lfloor 29n/180\rfloor,n/6}(\ol\zeta)))$ and $\wh c$, we conclude the proof.
    
\end{proof}

\section{Proof of Proposition~\ref{prop:discrete and continuous 1}}\label{sec4.2:proof of prop3.12}
We first introduce necessary notation (applying to $i=1,2$ respectively): 
\begin{itemize}
	\item Let $\ol\zeta = (\zeta_1,\zeta_2)$ be a pair of paths in $\mathrm{NICE}_{n/6}(z_n)$ with endpoints $x_1,x_2$. Write $\widetilde\zeta = (\widetilde\zeta_1,\widetilde\zeta_2)$ for their thickened counterparts.
	\item Let $S_i$ and $W_i$ be the simple random walk and Brownian motion starting from $x_i$ respectively.
	\item Let $\xi_1 = S_1[0,\tau_{n/6}]$, $\xi_2 = S_2[0,\tau_n]$ and $\widetilde\xi_1 = W_1[0,T_{n/6}]$, $\widetilde \xi_2 = W_2[0,T_n]$.
	\item Decompose $\xi_i$ as $\xi_i = \eta_i\oplus\eta_{i+2}$ where $\eta_i$ is the segment of $\xi_i$ from $x_i$ to its first exit of $\partial B(z_n,d_ze^{n-1})$.
	\item Decompose $\widetilde \xi_i$ as $\widetilde \xi_i = \widetilde\eta\oplus\widetilde\eta_{i+2}$ where $\widetilde\eta_i$ is the segment of $\widetilde \xi_i$ from $x_i$ to its first exit of $\partial D(z_n,d_ze^{n-1})$.
	\item Let $\Ec_n(\ol\zeta)$ denote the event that $z_n\in\fr((\zeta_1\cup\xi_1)\cup(\zeta_2\cup\xi_2))$.
	\item Let $\widetilde \Ec_n(\ol\zeta)$ denote the event that $z_n\in\fr((\widetilde\zeta_1\cup\widetilde \xi_1)\cup(\widetilde\zeta_2\cup\widetilde \xi_2))$.
\end{itemize}
The following lemma establishes up-to-constants estimates for the probabilities of events $\Ec_n(\ol\zeta)$ and $\widetilde\Ec_n(\ol\zeta)$, achieved through a path-decomposition methodology.
\begin{lemma}
	For all $\ol\zeta \in \mathrm{NICE}_{n/6}(z_n)$ with $d_z\ge e^{-n/6}$, we have
	\begin{equation}
		\Pb\big(\Ec_n(\ol\zeta)\big) \asymp a(z)n^{-1}e^{-5\alpha n/6},\quad \Pb\big(\widetilde\Ec_n(\ol\zeta)\big) \asymp a(z)n^{-1}e^{-5\alpha n/6},
	\end{equation}    
	where $a(z)$ defined in \eqref{eq:frontier green's function} is the order of $G^{\fr}_{\Dc}(z)$.
\end{lemma}

\begin{proof}
	We outline the lower bound; the upper bound follows similarly via decomposition. We first deal with the case $|z|>1/2$.
	\begin{itemize}
		\item The total mass of $(\eta_1,\eta_2)$ restricted to the event that $z_n\in\fr((\zeta_1\cup \eta_1)\cup(\zeta_2\cup \eta_2))$ and they are well-separated (see Lemma \ref{lem:rw-disc-sep}) at $\partial B(z_n,d_ze^{n-1})$ is of order $(d_ze^{n-1}/e^{n/6})^{-\alpha}$, noticing that the nice configuration $\ol\zeta$ is well-separated (see Lemma \ref{lem:rw-disc-sep}) at $\partial\Bc_{n/6}(z_n)$.
		\item The total mass of $\eta_4$ confined to a well-chosen tube of width $d_ze^{n-1}$ is of order 1.
		\item Let $\eta_3^1$ be the SRW started from the endpoint of $\eta_1$ that hits $\partial\Bc_{n-3}$ before exiting $\Bc_n$ such that $z_n\in\fr((\eta_3^1\cup\eta_1\cup\zeta_1)\cup(\zeta_2\cup\eta_2\cup\eta_4))$. By gambler's ruin estimate, the total mass is of order $\log({1}/{(1-d_z)})\asymp d_z$.
		\item Let $\eta_3^2$ be the SRW started from the endpoint of $\eta_3^1$ and hits $\partial\Bc_{n/6}$ before hitting $\partial \Bc_{n-2}$. The total mass of $\eta_3^2$ is of order $n^{-1}$ according to the gambler's ruin estimate.
	\end{itemize}
	The concatenation $(\eta_1\oplus \eta_3^1\oplus\eta_3^2,\eta_2\oplus\eta_4)$ satisfies the event $\Ec_n(\ol\zeta)$. By multiplying the probability measures derived above, we obtain
	\[
	\Pb\big(\Ec_n(\ol\zeta)\big) \asymp (d_ze^{n-1}/e^{n/6})^{-\alpha}\times 1 \times d_z \times n^{-1} \asymp  d_z^{1-\alpha}e^{-5\alpha n/6}n^{-1}.
	\]
	Next, we consider the case $|z|\le 1/2$.
	\begin{itemize}
		\item Let $(\lambda_1,\lambda_2)$ be a pair of non-disconnecting random walks from $(x_1,x_2)$ to the circle $\partial B(z_n,19d_z e^n/20)$ , satisfying the well-separation condition defined in \eqref{eq:disc-qua} with the endpoint of $\lambda_1$ lying in $B(0,d_z e^n/10)\cap\partial B(z_n,19d_z e^n/20)$. The total mass for such pairs $(\lambda_1,\lambda_2)$ is of order $d_z^{-\alpha} e^{-5\alpha n/6}$.
		\item Let $\lambda_3$ be a simple random walk starting from the endpoint of $\lambda_1$ that reaches $\partial \Bc_{n/6}$ before hitting $\partial B(z_n,d_ze^n/9)$. Using gambler's ruin estimates, the total mass of $\lambda _3$ is of order $[\log(d_ze^n/9)-\log(e^{n/6})]^{-1}\asymp n^{-1}$.
		\item Let $\lambda _4$ be a simple random walk starting from the endpoint of $\lambda_2$ that reaches $\partial \Bc_n$ before disconnecting $z_n$ and $\infty$. By Lemma \ref{one-arm disconnection exponent},  the total mass of $\lambda_4$ is of order $d_z^{1/4}$.
	\end{itemize}	
	The concatenation $(\eta_1\oplus\eta_3,\eta_2\oplus\eta_4)$ satisfies the event $\Ec_n(\ol\zeta)$. Multiplying the masses above yields 
	\[
	\Pb\big(\Ec_n(\ol\zeta)\big)\asymp d_z^{-\alpha} e^{-5\alpha n/6}\times n^{-1}\times d_z^{1/4}\asymp d_z^{1/4-\alpha}e^{-5\alpha n/6}n^{-1}.
	\]
    Similar arguments apply to $\Pb\big(\widetilde \Ec_n(\ol\zeta)\big)$. This completes the proof.
\end{proof}
Analogous to Proposition 7.9 in \cite{GLPS23}, we establish the following lemma, which implies Proposition~\ref{prop:discrete and continuous 1}.
\begin{lemma}
	For all $\ol\zeta\in\mathrm{NICE}_{n/6}(z_n)$ with $d_z\ge e^{-n/6}$,  the following equivalence holds:
	\begin{equation}
		\Pb\big(\Ec_n(\ol\zeta)\big)\simeq \Pb\big(\widetilde\Ec_n(\ol\zeta)\big).
	\end{equation}	
\end{lemma}

\end{document}